\documentclass[11pt]{article}
\usepackage{amsmath, amsthm, amssymb} 
\usepackage{hyperref}
\usepackage {mathrsfs, graphicx}

\newcommand{\be}{\begin{equation}}
\newcommand{\ee}{\end{equation}}
\newcommand{\bea}{\begin{eqnarray}}
\newcommand{\eea}{\end{eqnarray}}
\newcommand{\beas}{\begin{eqnarray*}}
\newcommand{\eeas}{\end{eqnarray*}}

\newcommand{\bbD}{\mathbb D}
\newcommand{\bbE}{\mathbb E}
\newcommand{\bbF}{\mathbb F}

\newcommand{\bbH}{\mathbb H}
\newcommand{\bbI}{\mathbb I}

\newcommand{\bbL}{\mathbb L}
\newcommand{\bbM}{\mathbb M}
\newcommand{\bbN}{\mathbb N}

\newcommand{\bbP}{\mathbb P}

\newcommand{\bbR}{\mathbb R}
\newcommand{\bbS}{\mathbb S}

\newcommand{\scA}{\mathcal A}
\newcommand{\scB}{\mathcal B}
\newcommand{\scC}{\mathcal C}
\newcommand{\scD}{\mathcal D}
\newcommand{\scE}{\mathcal E}
\newcommand{\scF}{\mathcal F}

\newcommand{\scH}{\mathcal H}

\newcommand{\scN}{\mathcal N}

\newcommand{\scP}{\mathcal P}

\newcommand{\scS}{\mathcal S}

\newcommand{\scV}{\mathcal V}

\newcommand{\scX}{\mathcal X}
\newcommand{\scY}{\mathcal Y}
\newcommand{\scZ}{\mathcal Z}

\newcommand{\veps}{\varepsilon}

\newcommand{\norm}[1]{\ensuremath{\left\| #1 \right\|}}
\newcommand{\abs}[1]{\ensuremath{\left| #1 \right|}}

\DeclareMathOperator*{\esssup}{ess\,sup}

\newcommand{\indicator}[1]{\ensuremath{\mathbf{1}_{#1}}}
\newcommand{\Ito}{It\ensuremath{\hat{\textrm{o}}}}
\newcommand{\half}{\frac{1}{2}}

\newcommand{\crl}[1]{\ensuremath{ \left\{ #1 \right\} }}
\newcommand{\edg}[1]{\ensuremath{ \left[ #1 \right] }}
\newcommand{\brak}[1]{\ensuremath{\left( #1 \right)}}

\newtheorem{theorem}{Theorem}[section]
\newtheorem{definition}[theorem]{Definition}
\newtheorem{proposition}[theorem]{Proposition}
\newtheorem{corollary}[theorem]{Corollary}
\newtheorem{lemma}[theorem]{Lemma}
\newtheorem{remark}[theorem]{Remark}
\newtheorem{example}[theorem]{Example}
\newtheorem{examples}[theorem]{Examples}
\newtheorem{foo}[theorem]{Remarks}
\newenvironment{Example}{\begin{example}\rm}{\end{example}}

\newenvironment{Remark}{\begin{remark}\rm}{\end{remark}}

\title{Locally Lipschitz BSDE driven by a continuous martingale\\{\Large path-derivative approach}}

\author{
Kihun Nam\\
Rutgers University\\
Piscataway, NJ 08854, USA
}
\begin{document}
\maketitle
\begin{abstract}
Using a new notion of path-derivative, we study well-posedness of backward stochastic differential equation driven by a continuous martingale $M$ when $f(s,\gamma,y,z)$ is locally Lipschitz in $(y,z)$: 
\[Y_{t}=\xi(M_{[0,T]})+\int_{t}^{T}f(s,M_{[0,s]},Y_{s-},Z_{s}m_{s})d{\rm tr}[M,M]_{s}-\int_{t}^{T}Z_{s}dM_{s}-N_{T}+N_{t}\]
Here, $M_{[0,t]}$ is the path of $M$ from $0$ to $t$ and $m$ is defined by $[M,M]_{t}=\int_{0}^{t}m_{s}m_{s}^{*}d{\rm tr}[M,M]_{s}$. When the BSDE is one-dimensional, we could show the existence and uniqueness of solution. On the contrary, when the BSDE is multidimensional, we show existence and uniqueness only when $[M,M]_{T}$ is small enough: otherwise, we provide a counterexample that has blowing-up solution. Then, we investigate the applications to utility maximization problems.
\\[2mm]
\end{abstract}
\setcounter{equation}{0}
\section{Introduction}
\label{sec:intro}
Let $M$ be a square integrable
continuous $n$-dimensional local martingale with quadratic
covariation matrix $[M,M]_{t}=\int_{0}^{t}m_{s}m_{s}^{*}d{\rm
  tr}[M,M]_{s}$ for a $\bbR^{n\times
  n}$-valued process $m$.
%We assume that $[N^{i},N^{j}]_{s}=\delta_{ij}K_{s}$. 
We let $D$ be the set of c\`adl\`ag $\bbR^{n}$-valued functions on
$[0,T]$. Consider the
following backward stochastic differential equation (BSDE) driven by $M$ where the \emph{terminal condition} is
$\xi:D\to\bbR^{d}$ and the \emph{driver} is
$
f:[0,T]\times D\times\bbR^{d}\times\bbR^{d\times n}\to\bbR^{d}
$:
\begin{align}\label{bsde}
Y_{t}=\xi(M_{[0,T]})+\int_{t}^{T}f(s,M_{[0,s]},Y_{s-},Z_{s}m_{s})d{\rm tr}[M,M]_{s}-\int_{t}^{T}Z_{s}dM_{s}-N_{T}+N_{t} 
\end{align}
Here, we denote $M_{[0,s]}$ to be the path of $M$ stopped at $s$. The \emph{solution} of above
BSDE is a triplet $(Y,Z,N)$ of adapted processes satisfying $[N,M]=0$.
We study the existence and uniqueness of solution when $\xi(\gamma)$ and $f(s,\gamma,y,z)$ are Lipschitz in $\gamma$ and locally Lipschitz in $(y,z)$. In order to do so, we study the differentiability of solutions under the perturbation of the path of $M$. Then, we apply our result to various utility maximization problems. 
%In order to achieve the well-posedness, we first find a priori estimate of the solution by exploiting the path-regularity of $\xi$ and $f$, and then prove the existence and uniqueness by localizing $f$ in bounded domain of $z$. Then, we use the martingale method in Hu et al. (2005, \cite{Hu:2005tj}) for utility maximization problem.

BSDE was first introduced by Bismut (1973, \cite{Bismut:1973vw})  as a dual problem of stochastic optimization under the assumptions $d=1$, Brownian motion $M$, and a linear function $f$. Then, Pardoux and Peng (1990, \cite{Pardoux:1990p20509}) extended the well-posedness result to $d\geq 1$ and Lipschitz function $f$. One can find classical results and applications in the survey paper written by El Karoui et al. (1997, \cite{ElKaroui:1997dn}). Since Pardoux and Peng's seminal paper, researchers extended the well-posedness result in various directions.

One direction of extension is to incorporate the case where $f$ grows superlinearly in $z$. The well-posedness results for such BSDEs have numerous applications including utility
maximization in incomplete market (Hu et al. 2005, \cite{Hu:2005tj}),
dynamic coherent risk measure (Gianin, 2006, \cite{RosazzaGianin:2006p26781}),
equilibrium pricing in incomplete market (Cheridito et al., 2016,
\cite{Cheridito:2009p29901}), and more recently, stochastic Radner equilibrium in
incomplete market (Kardaras et al., 2015, \cite{Kardaras:2015ii}). When $d=1$ and $\xi$ is bounded, Kobylanski (2000, \cite{Kobylanski:2000cy}) proved the existence and uniqueness of
solution when $f(s,\gamma,y,z)$ grows quadratically in $z$. Briand and Hu (2006, \cite{Briand:2006p17043}, 2008, \cite{Briand:2008p16923}), and Delbaen et al. (2011, \cite{Delbaen:2011ji}) further extended the results to unbounded terminal condition $\xi$. Superquadratic BSDE driven by a Brownian motion also attracted the interest among mathematician. Delbaen et al. (2010, \cite{Delbaen:2010gj}) showed that such BSDE is ill-posed if there is no regularity assumption on the terminal condition and the driver. Richou (2012, \cite{Richou:2012bf}) studied the existence and uniqueness of solution for superquadratic Markovian BSDE. Cheridito and Nam (2014, \cite{Cheridito:2014cb}) showed the existence and uniqueness of
solution for the non-Markovian case using Malliavin calculus and its connection to semilinear
parabolic PDEs under the Markovian assumption.

On the contrary, when $d>1$, Frei and dos Reis (2012, \cite{Anonymous:2012p19201})
showed that a multidimensional BSDE with a quadratic driver might not
be well-posed. By choosing a terminal condition which is
irregular with respect to the underlying Brownian motion, they were able
to construct an example such that the solution $Y$ blows up. When one does not assume regularity conditions on $\xi$ and $f$, only a few positive results are known when the terminal condition is small, or the driver satisfies certain restrictive structural conditions: see Tevzadze (2008, \cite{Tevzadze:2008p17046}),
Cheridito and Nam (2014, \cite{Cheridito:2013tg}), Hu and Tang (2014,
\cite{Hu:2014th}), Jamneshan et al, (2014, \cite{Jamneshan:2014ui}), Kupper et al. (2015, \cite{Kupper:2015wv}),  and
Xing and Zitkovic (2016, \cite{Xing:2016up}). When $\xi$ and $f$ are assumed to be regular, Nam(2014,\cite{Nam:2014wt}), Kupper et al. (2015, \cite{Kupper:2015wv}), and Cheridito and Nam (2017, \cite{Cheridito:2014uz})

Researchers also tried to generalize Brownian motion to a general martingale $M$. When $M$ is a continuous martingale, El Karoui and Huang (1997, \cite{ElKaroui:1997vv}) provided the existence and uniqueness of solution in the case where $f(t,\gamma_{[0,t]},y,z)$ is Lipschitz with respect to $(y,z)$ when $d\geq1$. When $d=1$, Morlais (2009, \cite{Morlais:2008jr}) investigated the existence and uniqueness of solution when $f(t,\gamma_{[0,t]},y,z)$ has quadratic growth in $z$.  Researchers generalized even to the case where $M$ is a general martingale with jumps. To name a few, Possamai et al. (2015, \cite{KaziTani:2015vd}) studied the case where $d=1$, $f$ has quadratic growth in $z$, and $M$ has jumps. On the other hand, Papapantoleon et al. (2016, \cite{Papapantoleon:2016ws}) treat the case where$d\geq 1$ and $f$ is (stochastically) Lipschitz in $(y,z)$. However, the following questions have not been answered when $M$ is a general martingale: 
\begin{itemize}
	\item If $d=1$, does one have well-posedness when $f(s,\gamma,y,z)$ grows superquadratically in $z$?
	\item If $d>1$, does one have well-posedness when $f(s,\gamma,y,z)$ grows superlinearly in $z$?
\end{itemize}

In this article, we answer these questions when $M$ is a continuous martingale; $\xi(\gamma)$ is Lipschitz with respect to $\gamma$; and $f(s,\gamma, y,z)$ is Lipschitz in $\gamma$ and locally Lipschitz in $(y,z)$. To be more specific, we were able to establish existence and uniqueness of solution when $d=1$ and find a uniform almost sure bound of the solution $Z$. In the case where $d>1$, we have the existence and uniqueness of solution as well as the bound of $Z$ only if $[M,M]_T$ is small enough: otherwise, we provide a counterexample such that $Z$ blows up.  We apply the 1D result to various kinds of control problems for SDE driven by $M$ using the martingale
method introduced by Hu et al. (2005, \cite{Hu:2005tj}). In the case where $M$ has jumps, our method does not work anymore,\footnote{see Remark \ref{jump}} and we leave this question for future papers.

The argument is based on the analysis of the stability under perturbation of $M$. We call this stability \textit{path-differentiability} of the solution. In other words, when we model stochastic optimization problem as a BSDE, the path-derivative of $Y$ implies the stability of value process with respect to the perturbation of underlying noise. An important property, which is called \textit{delta-hedging formula}, is that $Z$ is the path-derivative of $Y$ under appropriate notion if $M$ possesses martingale representation property. If one can find a uniform bound of $Z$ by estimating the derivative of $Y$ and using delta-hedging formula, we can use the localization argument to prove the well-posedness of BSDEs with locally Lipschitz drivers. 

Using Malliavin calculus on BSDE as in El Karoui et al. (1997, \cite{ElKaroui:1997dn}) and Hu et al. (2012, \cite{Nualart:2012p35991}), this strategy was used in Briand and Elie (2013, \cite{Briand:2013cu}), Cheridito and Nam (2014, \cite{Cheridito:2014cb}), and Kupper et al. (2015, \cite{Kupper:2015wv}) when $M$ is a Brownian motion. However, this method cannot be trivially extended to a continuous martingale $M$ because $M$ is not Malliavin differentiable in general. For example, consider the case where $M$ is a Brownian motion stopped at a hitting time. Even when $\xi$ is smooth, $\xi(W_{\tau})$ is not Malliavin
differentiable in general.\footnote{Cheridito and Nam (2014,
	\cite{Cheridito:2014cb}) If 
	$\tau$ is a stopping time such that $W_{\tau} $ is Malliavin differentiable, then $\tau$ must be a constant.
	Indeed, for $W_{\tau} = \int_0^{\infty} 1_{\crl{s < \tau}} dW_s \in \bbD^{1,2}$, one obtains from 
	Proposition 5.3 of El Karoui et al. (1997, \cite{ElKaroui:1997dn}) that $1_{\crl{s < \tau}} \in \bbD^{1,2}$ for almost 
	all $s$, and therefore, by Proposition 1.2.6 of Nualart (2006, \cite{Nualart:2006p36713}), $\bbP[s < \tau] = 0$ or $1$. } Therefore, classical Malliavin
calculus method used in the papers mentioned above cannot be used to study the path-differentiability of solution for this type of BSDE.\footnote{This type of BSDE is also known as BSDE with random terminal time and studied by numerous researchers including Darling and Pardoux (1997, \cite{Darling:1997ui}) and Jeanblanc et al. (2015, \cite{JEANBLANC:2015hf}).}

One may define another path-derivative notion for BSDE by assuming Markovian structure, that is one assumes $\xi(\gamma)=\xi(X_T)$ and $f(t,\gamma_{[0,T]},y,z)=f(t,X_t,y,z)$ where $dX_t=b(t,X_t)dA_t+\sigma(t,X_t)dM_t$ for some deterministic function $b$ and $\sigma$. In many cases, there is a deterministic measurable function $u$ such that $Y_t=u(t, X_t,M_t)$. Then one can define path-differentiability as a classical differentiability of the function $u$. This approach was used in Imkeller et al. (2012, \cite{Imkeller:2012dh}) to study existence, uniqueness, and path-differentiability of \eqref{bsde} with $d=1$ and $f$ grows quadratically in $z$. However, our problem deals with fully path-dependent $\xi$ and $f$, so this method also need to be extended to incorporate our problem.

One of the recent definitions of ``path-derivative" is the functional {\Ito} derivative developed by Dupire (2009, \cite{Dupire:2009eu}), and Cont and Fournie (2013,\cite{Cont:2013hy}). The perturbation in functional {\Ito} derivative is given by either horizontal or vertical displacement of the path at the last time. The functional {\Ito} calculus is general in a sense that it assume neither Markovian structure nor Gaussian property of $M$. Using vertical functional {\Ito} derivative, Cont (2016, \cite{Cont:2016wx}) was able to get a delta-hedging formula for \eqref{bsde} when $M$ is a continuous semimartingale determined by forward SDE driven by Brownian motion. Even though functional {\Ito} calculus has its own strength, it is not suitable for obtaining a uniform bound of $Z$. The reason is that we do not know the equation the functional {\Ito} derivative of $Y$ satisfies.

In order to find a uniform estimate of $Z$, we modify the vertical functional {\Ito} derivative to time-parametrized version similar to Malliavin derivative and use such notion to obtain BSDE for path-derivatives of $Y$ and $Z$. Then, by the classical method in BSDE, we get a uniform bound of $Z$. However, we should note that the path-functional representation of random variables and stochastic processes are not unique and our path-derivative definition crucially depends on the representation. Therefore, it is important to select logically consistent representations of the coefficients $\xi, f$ and our solution $Y,Z,N$. This is done by Theorem \ref{prelim2} and it is the main reason why we cannot extend our result to the case where $M$ has jumps.

The article is organized as follows. In Section \ref{sec:pre}, we
give the definitions, notations, and assumptions we use throughout this article. In
Section \ref{sec:basic}, we review the basic properties of BSDE with
Lipschitz driver and driven by a continuous martingale. Then we study
the differentiability of BSDE in Section \ref{sec:diffbsde}. Using results
from Section \ref{sec:basic} and \ref{sec:diffbsde}, we study the
existence and uniqueness of solution for BSDEs with locally
Lipschitz drivers in Section
\ref{sec:ext}. In particular, we show the existence and uniqueness
of solution when $[M,M]_{T}$ is small enough or $d=1$. Otherwise, the
solution may blow-up, and it is shown by an example in
subsection \ref{sec:counter}. Using the martingale method in Hu et al. (2005, \cite{Hu:2005tj}), we study utility maximization of controlled
SDE in Section \ref{sec:diff}. In Section \ref{sec:ops}, for power and exponential
utility function, the scheme is applied to optimal portfolio
selection under three different types of restriction: 1) when the
investment strategy is restricted to a closed set; 2) when the
diversification of portfolio gives the investor extra benefit; and 3)
when there is information processing cost for investment. 

\setcounter{equation}{0}
\section{Preliminaries}
\label{sec:pre}
{\noindent\bf Real space}

\smallskip
We denote $\bbR$ the set of real number and $\bbR_{+}$ the set of
nonnegative real numbers. For any natural numbers $l$ and $m$, $\bbR^{l\times m}$ is the set of real
$l$-by-$m$ matrices. $\bbR^{m}$ is the set of $m$-dimensional
real vectors and we identify with $\bbR^{m\times 1}$ unless
otherwise stated. For any matrix $X$,  we let $X^{*}$ to
be its transpose and we define $|X|$ to
be the Euclidean norm, that is
$|X|^{2}:={\rm tr}(XX^{*})$. We always endow Borel $\sigma$-algebra on
$\bbR^{l\times m}$ with respect to the norm $|\cdot|$ and denote it
by $\scB(\bbR^{l\times m})$. For $X\in\bbR^{l\times n}$, we denote
$(i,j)$-entry of $X$ as $X^{ij}$. We denote $\bbI$ to be the identity
matrix of appropriate size.

\bigskip
{\noindent\bf Probability space and the driving martingale}

\smallskip
Let $(\Omega,\scF,\bbF, \bbP)$ be a filtered probability space. 
We assume the filtration $\bbF:=(\scF_{t})_{t\in[0,T]}$ is complete,
quasi-left continuous,
and right continuous. Let $M$ be a square integrable
continuous $n$-dimensional martingale with a continuous
predictable quadratic covariation matrix $[M,M]$ and $M_{0}=0$. We assume that there exists a $\bbR^{n\times n}$-valued predictable process $m$ such that
\[[M,M]_{t}=\int_{0}^{t}m_{s}m_{s}^{*}dA_{s}\] 
where
\[
A_{t}:={\rm tr}[M,M]_{t}=\sum_{{i=1}}^{n}[M^{i},M^{i}]_{t}.
\]
Moreover, we always assume that $A_{T}$ is bounded by $K$. Then,
we have two consequences: 
\begin{itemize}
\item $|m_{s}|=1$ $ds\otimes d\bbP$-a.e. \phantom{XXX}($\because
  A_{t}=\sum_{i=1}^{n}\int_{0}^{t}\sum_{k=1}^{n}|m_{s}^{ik}|^{2}dA_{s}=\int_{0}^{t}|m_{s}|^{2}dA_{s}$)
\item $|[M,M]_{T}|\sim A_{T}$ \phantom{XXXXXX.}($\because A_{T}/\sqrt n\leq|[M,M]_{T}|=\abs{\int_{0}^{T}m_{s}m_{s}^{*}dA_{s}}\leq\int_{0}^{T}|m_{s}|^{2}dA_{s}=A_{T}$)
\end{itemize}

In addition, we assume there exists a Poisson random measure $\nu$ on $[0,T]\times \bbR^{n}$ with mean ${\rm Leb}\otimes \mu$ where
$\mu$ is the uniform probability measure on a unit ball centered at $0$. We let
\[
\hat M_{t}:=\int_{[0,t]\times\bbR^{n}}x\nu(ds,dx)
\]
and $\bbR^{n}$-valued c\`adl\`ag martingale $M':=M+\hat M$.  We
assume that $M$ and $\nu$ are independent and moreover, for any given
$\gamma\in\crl{\hat M(\omega):\omega\in\Omega}$ and $\omega'\in\Omega$, there is $\omega\in\Omega$ such that $\hat
M(\omega)=\gamma$ and $M(\omega)=M(\omega')$.
We also assume that $\bbF^{M}$, the augmentation of
$\sigma(M_{s}:s\leq t)$, is quasi-left continuous and right continuous.
This condition is true when $M$ is a Hunt
process: see Proposition 2.7.7 of Karatzas and Shreve (1991,
\cite{Karatzas:1991bi}) and Section 3.1 of Chung and Walsh (2005,
\cite{Chung:2005iu}). Therefore, any Feller process $M$ satisfies this
property. This implies, $\bbF^{M'}$, the augmentation of
$\sigma(M'_{s}:s\leq t)$, is also quasi-left continuous and right
continuous. It is noteworthy to observe $\bbF^{M}\subset\bbF^{M'}$
because $M$ is a continuous process while $\hat M$ is a pure jump process.

Note that since $\bbF$ is continuous, for any
$\bbR^{d}$-valued $(\bbF,\bbP)$-martingale is of
the form $\int ZdM+N$ where $Z$ is a $\bbF$-predictable $\bbR^{d\times
  n}$-valued process and $N$ is a $\bbR^{d}$-valued
$(\bbF,\bbP)$-martingale with $[N,M]=0$. This statement also holds
with $\bbF^{M'}$ or $\bbF^{M}$ instead of $\bbF$.

As always, we understand equalities and inequalities in
$\bbP$-almost sure sense.

\bigskip

{\noindent\bf The space of c\`adl\`ag paths}

\smallskip
We let $D$ be the set of all c\`adl\`ag $\bbR^{n}$-valued functions on
$[0,T]$. For $\gamma\in D$, we denote
$\gamma_{t}$ to be the value at time $t$ and $\gamma_{[0,t]}$ to be
the function $\gamma_{[0,t]}(s):=\gamma_{s\wedge t}$. For $\gamma,\gamma'\in D$, we define $(\gamma+\gamma')_t:=\gamma_t+\gamma'_t$. On $D$, we endow a sup norm,
$\norm{\gamma}_{\infty}:=\sup_{t\in[0,T]}|\gamma_t|$ and let $\scD$ be its Borel $\sigma$-algebra.  Then, we have the following lemma
whose proof is given at the appendix.
\begin{lemma}\label{msble}
A $\bbR^{k}$-valued stochastic process $X$ is adapted to $\bbF^{M'}$ if and only if there
exists a path functional $\scX:[0,T]\times D\to\bbR^{k}$ such that
\[
X_{t}=\scX(t,M'_{[0,t]})
\]
holds almost surely for each $t\in[0,T]$ and $\scX(t,\cdot)$ is $\scD$-measurable.
\end{lemma}
Let
$x_{t}(\gamma)=\gamma_{t}$ and define a filtration
$\scH_{t}:=\sigma(\crl {x_{s}:s\leq t})$. We let $\scP$ be the
predictable $\sigma$-algebra on $[0,T]\times D$ associated with
filtration $\crl{\scH_{t}}$. Then, it is easy to check that if a
function $f:[0,T]\times D\to\bbR^{d}$ is $\scP$-measurable, then
 $f(t,M_{[0,t]})$ is a predictable processe since
$M_{[0,\cdot]}:[0,T]\times\Omega\to D$ is a predictable processe.

%We assume that $[N^{i},N^{j}]_{s}=\delta_{ij}K_{s}$. 

\bigskip
{\noindent\bf Banach space}

\smallskip
We set the following Banach spaces:
\begin{itemize}
\item[$\bbL^{2}$:] all $d$-dimensional random vectors $X$ satisfying
$\norm{X}_2:=\sqrt{\bbE |X|^2} < \infty$
\item[$\bbS^{2}$:] all $\mathbb{R}^d$-valued c\`adl\`ag adapted processes
$(Y_t)_{0 \le t \le T}$ satisfying
$
\|Y\|_{\mathbb{S}^2}: = \norm{\sup_{0 \leq t \le T}|Y_{t}|}_2 < \infty
$
\item[$\bbH^{2}$:]all $\mathbb{R}^d$-valued c\`adl\`ag adapted processes
$(Y_t)_{0 \le t \le T}$ satisfying
$
\|Y\|_{\mathbb{H}^2}: = \bbE\int_{0}^{T}|Y_{t-}|^{2}dA_{t} < \infty
$
\item[$\bbH_{m}^{2}$:]all $\mathbb{R}^d$-valued predictable processes
$(Z_t)_{0 \le t \le T}$ satisfying
$
\|Z\|_{\mathbb{H}^2_{m}}: =  \bbE\int_{0}^{T}|Z_{t}m_{t}|^{2}dA_{t} < \infty
$
\item[$\bbM^{2}$:]all c\`adl\`ag martingale
$(N_t)_{0 \le t \le T}$ satisfying
$
\|N\|_{\mathbb{M}^2}: = \bbE {\rm tr}[N,N]_{T} < \infty
$, $ [N,M]=0 $, and $N_{0}=0$.
\end{itemize}

\bigskip
{\noindent\bf BSDE and its solution}

\smallskip
Assume that $\xi:D\to\bbR^{d}$ is $\scD$-measurable and $f:[0,T]\times
D\times\bbR^{d}\times\bbR^{d\times n}\to\bbR^{d}$ is
$\scP\otimes\scB(\bbR^{d})\otimes\scB(\bbR^{d\times n})$-measurable. The \emph{solution} of BSDE($\xi,f$) is a triplet of adapted processes
$(Y,Z,N)\in\bbH^{2}\times\bbH^{2}_{m}\times\bbM^{2}$ which satisfies
\begin{equation}\tag{BSDE($\xi,f$)}
Y_{t}=\xi(M_{[0,T]})+\int_{t}^{T}f(s,M_{[0,s]},Y_{s-},Z_{s}m_{s})dA_{s}-\int_{t}^{T}Z_{s}dM_{s}-N_{T}+N_{t}.
\end{equation}
With a slight abuse of notation, sometimes we denote the above BSDE
as BSDE($\xi(M_{[0,T]}),f$).

\section{Properties of BSDE with Lipschitz
  Driver}\label{sec:basic}
In this section we present existence, uniqueness, stability, comparison, and path-representation results regarding
BSDE($\xi,f$) when $f(s,\gamma,y,z)$ is Lipschitz with respect to $(y,z)$. Except for path-representation of solutions provided in Theorem \ref{prelim} and Theorem \ref{prelim2}, most of results are well-known: see e.g. El Karoui and Huang (1997, \cite{ElKaroui:1997vv}). However, we provide the proof for the readers' convenience. Set:
\begin{itemize}
\item[(STD)] The terminal condition $\xi$ is in $\bbL^{2}$. Let
  $\scP'$ be the progressively measurable $\sigma$-algebra on
  $[0,T]\times \Omega$. The driver
  $f:[0,T]\times\Omega\times\bbR^{d}\times\bbR^{d\times n}\to\bbR^{d}$ is
  $\scP'\otimes\scB(\bbR^{d})\otimes \scB(\bbR^{d\times
    n})$-measurable function such that
  $\bbE\int_{0}^{T}|f(s,0,0)|^{2}dA_{s}<\infty$. Moreover, we assume
  that there are $C_{y},C_{z}\in\bbR_{+}$ such that
\[
|f(s,y,z)-f(s,y',z')|\leq C_{y}|y-y'|+C_{z}|z-z'|.
\]
\end{itemize}
\begin{proposition}\label{stdbsde}
Assume {\rm (STD)}. Then there exists a unique solution
$(Y,Z,N)\in\bbH^{2}\times\bbH^{2}_{m}\times\bbM^{2}$ of
\begin{align}\label{stdBSDEeq}
Y_{t}=\xi+\int_{t}^{T}f(s,Y_{s-},Z_{s}m_{s})dA_{s}-\int_{t}^{T}Z_{s}dM_{s}-N_T+N_t 
\end{align}
and moreover, $Y\in\bbS^{2}$.
\end{proposition}
\begin{Remark}
Note that $f(s,Y_{s-},Z_{s})$ may not be predictable. Therefore, the
integral with respect to $A$ should be
interpreted as Lebesgue-Stieltjes integral.
\end{Remark}

\begin{proof}
Let us define the following Banach
  spaces:
\begin{itemize}
\item[$\bbH_{a}^{2}$:]all $\mathbb{R}^d$-valued c\`adl\`ag adapted processes
$(Y_t)_{0 \le t \le T}$ satisfying
$
\|Y\|_{\mathbb{H}_{a}^2}: =  \bbE\int_{0}^{T}e^{aA_{t}}|Y_{t-}|^{2}dA_{t} < \infty.
$

\item[$\bbH_{m,a}^{2}$:]all $\mathbb{R}^d$-valued predictable processes
$(Z_t)_{0 \le t \le T}$ satisfying
$
\|Z\|_{\mathbb{H}^2_{m,a}}: =  \bbE\int_{0}^{T}e^{aA_{t}}|Z_{t}m_{t}|^{2}dA_{t} < \infty.
$
\item[$\bbM^{2}_{a}$:]all c\`adl\`ag martingale
$(N_t)_{0 \le t \le T}$ satisfying
$
\|N\|_{\mathbb{M}_{a}^2}: = \bbE\int_{0}^{T}e^{aA_{s}}d{\rm
  tr}[N,N]_{t} < \infty$ and $N_{0}=0$.
\end{itemize}
For $a=2\abs{C_{y}\vee C_{z}}^{2}+2$, we will use contraction mapping
theorem for \[\phi:(y,z,n)\in\bbH_{a}^{2}\times \bbH_{m,a}^{2}\times
\bbM_{a}^{2}\mapsto (Y,Z,N)\in\bbH_{a}^{2}\times \bbH^{2}_{m,a}\times
\bbM_{a}^{2}.\] where $(Y,Z,N)$ is
 given by the solution of BSDE
\[
Y_{t}=\xi+\int_{t}^{T}f(s,y_{s-},z_{s}m_{s})dA_{s}-\int_{t}^{T}Z_{s}dM_{s}-N_{T}+N_{t},
\]
or equivalently, 
\[
Y_{0}+\int_{0}^{t}Z_{s}dM_{s}+N_{t}=\bbE_t\edg{\xi+\int_{0}^{T}f(s, y_{s-},z_{s}m_{s})dA_{s}}
\]
Then, since
$e^{aA_{s}}$ is between $1$ and $e^{aK}$, the space $\bbH_{a}^{2}\times \bbH_{m,a}^{2}\times
\bbM_{a}^{2}$ is equivalent to
$\bbH^{2}\times\bbH_{m}^{2}\times\bbM^{2}$ and the fixed point we get by
contraction mapping theorem is
the unique solution in $\bbH^{2}\times\bbH^{2}_{m}\times\bbM^{2}$.

First, let us show that $\phi(y,z,n)=(Y,Z,N)$ is in
$\bbH^{2}\times\bbH_{m}^{2}\times\bbM^{2}$ and therefore in $\bbH_{a}^{2}\times \bbH_{m,a}^{2}\times
\bbM_{a}^{2}$. From Theorem 27 Corollary 3 of II.6 of
Protter (2004, \cite{Protter:2004wf}),
\begin{align*}
\bbE\abs{\bbE_t\edg{\xi+\int_{0}^{T}f(s, y_{s-},z_{s}m_{s})dA_{s}}}^{2} \leq\bbE\abs{\xi+\int_{0}^{T}f(s, y_{s-},z_{s}m_{s})dA_{s}}^{2}<\infty
\end{align*}
for all $t\in[0,T]$ implies
\begin{align*}
\norm{Z}^{2}_{\bbH_{m}^{2}}+\norm{N}^{2}_{\bbM^{2}}=\bbE{\rm tr}\edg{\int_{0}^{\cdot}Z_{s}dM_{s}+N,\int_{0}^{\cdot}Z_{s}dM_{s}+N}_{T}<\infty.
\end{align*}
On the other hand,
\begin{align*}
|Y_{t}|\leq |\xi|+\int_{0}^{T}|f(s,y_{s-},z_{s}m_{s})|dA_{s}+\sup_{t\in[0,T]}\abs{\int_{t}^{T}Z_{s}dM_{s}+N_{T}-N_{t}}.
\end{align*}
From Burkholder-Davis-Gundy inequality, for some constant $C'$,
\begin{align*}
\bbE\sup_{t\in[0,T]}\abs{\int_{t}^{T}Z_{s}dM_{s}+N_{T}-N_{t}}^{2}&\leq
  2\bbE\abs{\int_{0}^{T}Z_{s}dM_{s}+N_{T}}^{2}+2\bbE\sup_{t\in[0,T]}\abs{\int_{0}^{t}Z_{s}dM_{s}+N_{t}}^{2}\leq C'\brak{\norm{Z}_{\bbH_{m}^{2}}^{2}+\norm{N}_{\bbM^{2}}^{2}}\end{align*}
Therefore, $Y\in\bbS^{2}$ and this implies $Y\in\bbH^{2}$.
% By {\Ito} formula,
% \begin{align*}
% |Y_{0}|^{2}
% &=e^{2A_{T}}|\xi|^{2}+\int_{0}^{T}e^{2A_{s}}\brak{2Y^{*}_{s}f(s,y_{s},z_{s}m_{s})-|Z_{s}m_{s}|^{2}-2|Y_{s}|^{2}}dA_{s}-\int^{T}_{0}e^{2A_{s}}d{\rm tr}\edg{N,N}_{s}\\ 
% &\qquad-2\int_{0}^{T}e^{2A_{s}}
%   Y_{s}^{*}Z_{s}dM_{s}-2\int_{0}^{T}e^{2A_{s}}Y_{s}^{*}dN_{s}\\
% &\leq e^{2A_{T}}|\xi|^{2}+\int_{0}^{T}e^{2A_{s}}\brak{|f(s,y_{s},z_{s}m_{s})|^{2}-|Z_{s}m_{s}|^{2}-|Y_{s}|^{2}}dA_{s}-\int^{T}_{0}e^{2A_{s}}d{\rm tr}\edg{N,N}_{s}\\
% &\qquad-2\int_{0}^{T}e^{2A_{s}} Y_{s}^{*}Z_{s}dM_{s}-2\int_{0}^{T}e^{2A_{s}}Y_{s}^{*}dN_{s}.
% \end{align*}
% Therefore, if we take expectation on both side, we get
% \begin{align*}
% \bbE\int_{0}^{T}e^{2A_{s}}\brak{|Y_{s}|^{2}+|Z_{s}m_{s}|^{2}} dA_{s}+\bbE \int^{T}_{0}e^{2A_{s}}d{\rm tr}\edg{N,N}_{s}&\leq \bbE e^{aA_{T}}|\xi|^{2}+\bbE\int_{0}^{T}e^{2A_{s}}|f(s,y_{s},z_{s}m_{s})|^{2}dA_{s}<\infty
% \end{align*}
%Since $0\leq A_{s}\leq K$ for $s\in[0,T]$, we proved the claim.

Next, let us show the contraction. Let $(Y,Z,N):=\phi(y,z,n)$ and
$(Y',Z',N'):=\phi(y',z',n')$. Let us denote $\delta
Y_{s}:=Y_{s}-Y'_{s},\delta Z_{s}:=Z_{s}-Z'_{s}, \delta N_{s}:=N_{s}-N'_{s}$, and $\delta
f_{s}:=f(s,y_{s-},z_{s}m_{s})-f(s,y'_{s-},z'_{s}m_{s})$. Then,
\begin{align*}
\delta Y_{t}=\int_{t}^{T}\delta f_{s} dA_{s}-\int_{t}^{T}\delta
  Z_{s}dM_{s}-\delta N_{T}+\delta N_{t}.
\end{align*}
By {\Ito} formula, 
\begin{align*}
0\leq |\delta Y_{0}|^{2}&=\int_{0}^{T}e^{aA_{s}}(2\delta Y_{s-}^{*}\delta
  f_{s}-|\delta Z_{s}m_{s}|^{2}-a|\delta
                    Y_{s-}|^{2})dA_{s}-\int^{T}_{0}e^{aA_{s}}d{\rm
                    tr}\edg{\delta N,\delta N}_{s}\\
&\qquad-2\int_{0}^{T}e^{2A_{s}} \delta Y_{s-}^{*}\delta
  Z_{s}dM_{s}-2\int_{0}^{T}e^{2A_{s}}\delta Y_{s-}^{*}d\delta N_{s}\\
&\leq \int_{0}^{T}e^{aA_{s}}((a-1)|\delta Y_{s-}|^{2}+\frac{1}{a-1}|\delta
  f_{s}|^{2}-|\delta Z_{s}m_{s}|^{2}-a|\delta
                    Y_{s-}|^{2})dA_{s}-\int^{T}_{0}e^{aA_{s}}d{\rm
                    tr}\edg{\delta N,\delta N}_{s}\\
&\qquad-2\int_{0}^{T}e^{2A_{s}} \delta Y_{s-}^{*}\delta
  Z_{s}dM_{s}-2\int_{0}^{T}e^{2A_{s}}\delta Y_{s-}^{*}d\delta N_{s}.
\end{align*}
If we take expectation on both side and rearrange it, by Lemma \ref{mtg}, we get
\begin{align*}
\bbE\int_{0}^{T}e^{aA_{s}}\brak{|\delta Y_{s-}|^{2}+|Z_{s}m_{s}|^{2}}&dA_{s}+\bbE \int^{T}_{0}e^{aA_{s}}d{\rm
                    tr}\edg{\delta N,\delta N}_{s}\\
&\leq
  \frac{1}{a-1}\bbE\int_{0}^{T}e^{aA_{s}}|\delta f_{s}|^{2}dA_{s}\leq\frac{a-2}{a-1}\bbE\int_{0}^{T}e^{aA_{s}}\brak{|
  y_{s-}-y'_{s-}|^{2}+|z_{s}m_{s}-z'_{s}m_{s}|^{2}}dA_{s}
\end{align*}
and $\phi$ is a contraction on
$\bbH^{2}_{a}\times\bbH^{2}_{m,a}\times\bbM^{2}_{a}$.
Therefore, there
exists a unique fixed point, which is our solution, in
$\bbH^{2}_{a}\times\bbH^{2}_{m,a}\times\bbM^{2}_{a}$. Therefore, there is a unique solution in
$(Y,Z,N)\in\bbH^{2}\times\bbH_{m}^{2}\times\bbM^{2}$ and
$Y\in\bbS^{2}$ from the argument at the beginning of the proof.
\end{proof}

\begin{proposition}\label{bddY}
Assume {\rm (STD)}. Moreover, assume that there exist
$C_{\xi},C_{f}\in\bbR_{+}$ such that
$|\xi|\leq C_{\xi}$
and $\int_{0}^{T}|f(s,0,0)|^{2}dA_{s}\leq C_{f}^{2}$. Then,
for solution $(Y,Z,N)$ of \eqref{stdBSDEeq}, we have
$
|Y_{t}|\leq \sqrt{C_{\xi}^{2}+ C_{f}^{2}}e^{\half K(2C_{y}+C_{z}^{2}+1)}
$.
\end{proposition}
\begin{proof}
 Since (STD) are satisfied, there exists a unique solution
$(Y,Z,N)\in\bbH^{2}\times\bbH_{m}^{2}\times\bbM^{2}$. By {\Ito} formula,
when $a=2C_{y}+C_{z}^{2}+1$,
we have
\begin{align*}
e^{aA_{t}}|Y_{t}|^{2}
&=e^{aA_{T}}|\xi|^{2}+\int_{t}^{T}e^{aA_{s}}\brak{2Y^{*}_{s-}f(s,Y_{s-},Z_{s}m_{s})-|Z_{s}m_{s}|^{2}-a|Y_{s-}|^{2}}dA_{s}-\int^{T}_{t}e^{aA_{s}}d{\rm tr}\edg{N,N}_{s}\\ 
&\qquad-2\int_{t}^{T}e^{aA_{s}}
  Y_{s-}^{*}Z_{s}dM_{s}-2\int_{t}^{T}e^{aA_{s}}Y_{s-}^{*}dN_{s}\\
&\leq e^{aA_{T}}|\xi|^{2}+\int_{0}^{T}e^{aA_{s}}|f(s,0,0)|^{2}dA_{s}-2\int_{t}^{T}e^{aA_{s}} Y_{s-}^{*}Z_{s}dM_{s}-2\int_{t}^{T}e^{aA_{s}}Y_{s-}^{*}dN_{s}.
\end{align*}
because
\begin{align*}
2Y^{*}_{s-}f(s,Y_{s-},Z_{s}m_{s})\leq |f(s,0,0)|^{2}+(2C_{y}+C_{z}^{2}+1)|Y_{s-}|^{2}+|Z_{s}m_{s}|^{2}.
\end{align*}
If we take $\bbE(\cdot |\scF_{t})$ on both side, by Lemma \ref{mtg}, we get
\begin{align*}
|Y_{t}|^{2}\leq\bbE_t\edg
{  e^{aA_{T}}|\xi|^{2}}+\bbE_t\edg{e^{aA_{T}}\int_{0}^{T}|f(s,0,0)|^{2}dA_{s}}\leq
(C_{\xi}^{2}+C_{f}^{2})  e^{K(2C_{y}+C_{z}^{2}+1)}
\end{align*}
\end{proof}
\begin{proposition}\label{prop:stab}(Stability) Assume that $(\xi,f)$
  satisfies {\rm (STD)} with Lipschitz coefficients $C_{y}$ and $C_{z}$. Also
  assume that $(\bar\xi,\bar f)$ satisfies {\rm (STD)} possibly with different Lipschitz coefficients. Let $(Y,Z,N)$ and
  $(\bar Y,\bar Z,\bar N)$ are solutions of 
\begin{align*}
Y_{t}&=\xi+\int_{t}^{T}f(s,
  Y_{s-},Z_{s}m_{s})dA_{s}-\int_{t}^{T}Z_{s}dM_{s}-N_{T}+N_{t}\\
\bar Y_{t}&=\bar\xi+\int_{t}^{T}\bar f(s, \bar Y_{s-},\bar Z_{s}m_{s})dA_{s}-\int_{t}^{T}\bar Z_{s}dM_{s}-\bar N_{T}+\bar N_{t},
\end{align*}
Then, we have the following estimate:
\begin{align*}
\norm{Y_{t}-\bar Y_{t}}^{2}_{2}+\norm{Y-\bar Y}^{2}_{\bbH^{2}}+\norm{Z-\bar
  Z}^{2}_{\bbH_{m}^{2}}&+\norm{N-\bar N}^{2}_{\bbM^{2}}\\
&\leq
  2e^{K(2C_{y}+2C_{z}^{2}+2)}\brak{\norm{\xi-\bar\xi}^{2}_{2}+\norm{f(\cdot,\bar
  Y_{\cdot},\bar Z_{\cdot}m_{\cdot})- \bar f(\cdot,\bar
  Y_{\cdot},\bar Z_{\cdot}m_{\cdot})}^{2}_{\bbH^{2}}}
\end{align*}
\end{proposition}
\begin{proof}
Denote $\delta Y:=Y-\bar Y,\;
\delta Z:=Z-\bar Z,\;
\delta N:=N-\bar N,\;\delta\xi:=\xi-\bar\xi$, and 
\begin{align*}
g(s,y,z):=f(s,\bar Y_{s-}+y,\bar Z_{s}m_{s}+z)-\bar f(s,\bar Y_{s-},\bar Z_{s}m_{s}).
\end{align*}
Then, we have
\begin{align*}
\delta Y_{t}=\delta\xi+\int_{t}^{T}g(s,\delta Y_{s-},\delta Z_{s}m_{s}) dA_{s}-\int_{t}^{T}\delta
  Z_{s}dM_{s}-\delta N_{T}+\delta N_{t}.
\end{align*}
where $(\delta\xi, g)$ satisfies (STD).
By applying {\Ito}
formula on $e^{aA_{t}}|\delta Y_{t}|^{2}$ where $a=2C_{y}+2C^{2}_{z}+2$, we have
\begin{align*}
 e^{aA_{t}}|\delta
  Y_{t}|^{2}=e^{aA_{T}}|\delta\xi|^{2}&+\int_{t}^{T}e^{aA_{s}}(2\delta Y_{s-}^{*}g(s,\delta Y_{s-},\delta Z_{s}m_{s})-a|\delta Y_{s-}|^{2}-|\delta
  Z_{s}m_{s}|^{2})dA_{s}\\
&-\int_{t}^{T}e^{aA_{s}}d{\rm tr}[\delta
  N,\delta N]_{s}
-\int_{t}^{T}2e^{aA_{s}}\delta
  Y^{*}_{s-}Z_{s}dM_{s}-\int_{t}^{T}2e^{aA_{s}}\delta Y^{*}_{s-}dN_{s}.
\end{align*}
and this implies, by Lemma \ref{mtg},
\begin{align*}
\bbE e^{aA_{t}}|\delta
  Y_{t}|^{2}&+\bbE\int_{t}^{T}e^{aA_{s}}|\delta
  Y_{s-}|^{2}dA_{s}+\half\bbE\int_{t}^{T}e^{aA_{s}}|\delta
  Z_{s}m_{s}|^{2}dA_{s} +\bbE \int_{t}^{T}e^{aA_{s}}d{\rm tr}[\delta
  N,\delta N]_{s}\\
&\leq
\bbE  e^{aA_{T}}|\delta\xi|^{2}+\bbE\int_{t}^{T}e^{aA_{s}}|g(s,0,0)|^{2}dA_{s}\leq e^{K(2C_{y}+2C_{z}^{2}+2)}\brak{\bbE |\delta\xi|^{2}+\bbE\int_{0}^{T}|g(s,0,0)|^{2}dA_{s}}
\end{align*}
since
\begin{align*}
2\delta Y_{s-}^{*}(g(s,\delta Y_{s-},\delta Z_{s}m_{s}))&\leq
                                                                  |g(s,0,0)|^{2}+(2C_{y}+2C_{z}^{2}+1)|\delta
               Y_{s-}|^{2}+\half|\delta Z_{s}m_{s}|^{2}.
\end{align*}
Therefore,
\begin{align*}
\sup_{t\in[0,T]}\bbE|\delta
  Y_{t}|^{2}&+\bbE\int_{0}^{T}|\delta
  Y_{s-}|^{2}dA_{s}+\bbE\int_{0}^{T}|\delta
  Z_{s}m_{s}|^{2}dA_{s} +\bbE {\rm tr}[\delta
  N,\delta N]_{T}\\
&\leq 2\brak{ \sup_{t\in[0,T]}\bbE e^{aA_{t}}|\delta
  Y_{t}|^{2}+\bbE\int_{0}^{T}e^{aA_{s}}|\delta
  Y_{s-}|^{2}dA_{s}+\half\bbE\int_{0}^{T}e^{aA_{s}}|\delta
  Z_{s}m_{s}|^{2}dA_{s} +\bbE \int_{0}^{T}e^{aA_{s}}d{\rm tr}[\delta
  N,\delta N]_{s}}\\
&\leq 2e^{K(2C_{y}+2C_{z}^{2}+2)}\brak{\bbE |\delta\xi|^{2}+\bbE\int_{0}^{T}|g(s,0,0)|^{2}dA_{s}}.
\end{align*}
% Here we define
% \begin{align*}
% P^{i}_{s}&:=\frac{f(s,M_{[0,s]},Y^{bar, i}_{s},Z_{s})-
%   \bar f(s,M_{[0,s]},Y^{bar, i+1}_{s},Z_{s})}{\delta Y^{i}_{s}}\\
% Q^{ij}_{s}&:=\frac{f\brak{s,M_{[0,s]},\bar
%             Y_{s},Z^{bar, ij}_{s}}-
%   \bar f\brak{s,M_{[0,s]},\bar Y_{s}, Z^{bar,i(j+1)}_{s}}}{\delta Z^{ij}_{s}}
% \end{align*}
% where
% \begin{align*}
% Y^{bar, i}&:=(\bar Y^{1},\cdots,\bar Y^{i-1}, Y^{i},Y^{i+1},\cdots Y^{d})\\
% Z^{bar, ij}&:=\begin{pmatrix}\bar Z^{11}&&&\cdots&&&\bar
%               Z^{1n}\\\vdots&&&&&&\vdots \\\bar
%               Z^{i1}&\cdots&\bar Z^{i(j-1)}&Z^{ij}&Z^{i(j+1)}&\cdots&Z^{in} \\\vdots&&&&&&\vdots\\ Z^{d1}&&&\cdots&&&Z^{dn}\end{pmatrix}
% \end{align*}
% and we define $\frac{0}{0}=0$. 
\end{proof}
Now let us prove the comparison theorem when $d=1$. This result will be
used in Section 4.2. We will denote $\scE(X)=\exp\brak{X-\half[X,X]}$. 
%Before we prove it, we need the Girsanov theorem
%for continuous local martingale.
%
%\begin{theorem}[Kazamaki (1994, \cite{Kazamaki:1994ux})]
%For a continuous local martingale $X$, assume that $\scE(X)=\exp\brak{X-\half[X,X]}$ is a uniformly integrable martingale. Let $d\tilde\bbP:=\scE(X)_{T}d\bbP$.
%For any continuous $\bbP$-local martingale $M$, the process $\tilde M:=M-[X,M]$
%is a continuous $\tilde\bbP$-local martingale and $[M,M]=[\tilde
%M,\tilde M]$ under either probability measure.
%\end{theorem}
\begin{theorem}\label{compthm}(Comparison Theorem) Let $d=1$ and assume that $m_{t}$
  is invertible for all $t\in[0,T]$.
Assume {\rm (STD)} for $(\xi, f)$ and $(\bar \xi,\bar  f)$. Let
$(Y,Z,N)$ and $(\bar  Y,\bar  Z,\bar  N)$ be the solutions of
\begin{align*}
Y_{t}=\xi+\int_{t}^{T}f(s,Y_{s-},Z_{s}m_{s})dA_{s}-\int_{t}^{T}Z_{s}dM_{s}-N_{T}+N_{t}\\
\bar  Y_{t}=\bar \xi+\int_{t}^{T}\bar  f(s, \bar  Y_{s-},\bar  Z_{s}m_{s})dA_{s}-\int_{t}^{T}\bar  Z_{s}dM_{s}-\bar  N_{T}+\bar  N_{t}.
\end{align*}
If $\xi\leq \bar \xi$
almost surely
and $f(t,y,z)\leq \bar  f(t,y,z)$ $dt\otimes
d\bbP\otimes dy\otimes dz$-almost everywhere, then $Y_{t}\leq \bar
Y_{t}$ a.s. for all $t\in[0,T]$
\end{theorem}
\begin{proof}
Let us denote \begin{align*}
\delta\xi&:=\xi-\bar \xi&\delta f_{s}&:=f(s,\bar  Y_{s-},\bar  Z_{s}m_{s})-\bar  f(s,\bar  Y_{s-},\bar  Z_{s}m_{s})\\
\delta Y_{s}&:=Y_{s}-\bar  Y_{s}, &\delta Z_{s}&:=Z_{s}-\bar  Z_{s},\qquad\qquad \qquad\qquad \qquad\qquad\delta
                                                           N_{s}:=N_{s}-\bar  N_{s}
\end{align*}
and
\begin{align*}
(Zm)^{(i)}&:=((\bar  Zm)^{1},(\bar Zm)^{2},\cdots, (\bar Zm)^{i},(Zm)^{i+1},(Zm)^{i+2},\cdots, (Zm)^{n})\\
F_{s}&:=\frac{f(s,Y_{s-},Z_{s}m_{s})-f(s,\bar  Y_{s-},Z_{s}m_{s})}{\delta Y_{s-}},\qquad
G^{i}_{s}:=\frac{f(s,\bar  Y_{s-},(Zm)^{(i-1)}_{s})-f(s,\bar  Y_{s-},(Zm)^{(i)}_{s})}{
                                (\delta Zm)_{s}^{i}}\\
%\tilde M&:=M-\edg{\int_{0}^{\cdot}G^{*}_{s}(m_{s})^{-1}dM_{s},M}\\
d\Gamma_{s}&:=\Gamma_{s}\brak{F_{s}dA_{s}+G^{*}_{s}(m_{s})^{-1}dM_{s}};\qquad\Gamma_{0}=1.
\end{align*}
Note that $F$ and $G^{i}$ are uniformly bounded by $C_{y}$ and
$C_{z}$, respectively, and
$\Gamma_{t}\geq 0$ for all $t$. Moreover, 
\begin{align*}
\Gamma_{t}&=\scE\brak{\int_{0}^{\cdot}F_{s}dA_{s}+\int_{0}^{\cdot}G^{*}_{s}(m_{s})^{-1}dM_{s}}_{t}\leq e^{C_yK}\scE\brak{\int_{0}^{\cdot}G^{*}_{s}(m_{s})^{-1}dM_{s}}_{t}
\end{align*}
where $\scE\brak{\int_{0}^{t}G^{*}_{s}(m_{s})^{-1}dM_{s}}$ is a
martingale because of Novikov condition. Note that, by Doob's maximal inequality,
\begin{align*}
\bbE\sup_{0\leq s\leq t}
  \abs{\scE\brak{\int_{0}^{\cdot}G^{*}_{s}(m_{s})^{-1}dM_{s}}_{s}}^{2}&\leq
                                                                        4\bbE
                                                                        \abs{\scE\brak{\int_{0}^{\cdot}G^{*}_{s}(m_{s})^{-1}dM_{s}}_{t}}^{2}\\
&\leq4\bbE
  e^{\int_{0}^{t}|G_{s}|^{2}dA_{s}}\scE\brak{2\int_{0}^{\cdot}G^{*}_{s}(m_{s})^{-1}dM_{s}}_{t}\leq 4e^{nC_{z}^{2} K}<\infty
\end{align*}
and therefore, $\Gamma\in\bbS^{2}$.

On the other hand, if we subtract both equations, we get
\begin{align*}
\delta Y_{t}=\delta \xi+\int_{t}^{T}\brak{\delta f_{s}+F_{s}\delta
  Y_{s-}+\delta Z_{s}m_{s}G_{s}}dA_{s}-\int_{t}^{T}\delta
  Z_{s}dM_{s}-\delta N_{T}+\delta N_{s}.
\end{align*}
If we apply
{\Ito} formula to $\Gamma_{s}\delta Y_{s}$, we get
\begin{align}\label{lbsde}
\Gamma_{t}\delta Y_{t}=\delta Y_{0}-\int_{0}^{t}\Gamma_{s}\delta
  f_{s}dA_{s}+\int_{0}^{t}(\delta
  Y_{s-}\Gamma_{s}G^{*}_{s}(m_{s})^{-1}+\Gamma_{s}\delta
  Z_{s})dM_{s}+\int_{0}^{t}\Gamma_{s}d\delta N_{s}
\end{align}
This implies $\Gamma \delta Y+\int\Gamma_{s}\delta f_{s}dA_{s}$ is a
local martingale. Note that
\begin{align*}
\bbE\sup_{0\leq s\leq t} \abs{\Gamma_{s}\delta
  Y_{s}}&\leq\half \norm{\Gamma}^{2}_{\bbS^{2}}+\half\norm{\delta Y}^{2}_{\bbS^{2}}\\
\bbE\sup_{0\leq s\leq t} \abs{\int_{0}^{s}\Gamma_{u}\delta f_{u}dA_{u}}&\leq \bbE\sup_{0\leq s\leq t}|\Gamma_{s}| \int_{0}^{t}|\delta
  f_{u}|dA_{u}\leq \half \norm{\Gamma}^{2}_{\bbS^{2}}+\half\bbE\brak{\int_{0}^{t}|\delta
  f_{u}|dA_{u}}^{2}<\infty
\end{align*}
Therefore, $\Gamma \delta Y+\int\Gamma_{s}\delta f_{s}dA_{s}$ and
$
\int(\delta
  Y_{s-}\Gamma_{s}G^{*}_{s}(m_{s})^{-1}+\Gamma_{s}\delta Z_{s})dM_{s}+\int\Gamma_{s}dN_{s}
$
are martingales. 
% Since
% \begin{align*}
% \edg{\int_{0}^{\cdot}G^{*}_{s}(m_{s})^{-1}dM_{s},\int_{0}^{\cdot}G^{*}_{s}(m_{s})^{-1}dM_{s}}_{T}\leq
%   \int_{0}^{T}|G_{s}|^{2}dA_{s}\leq nC^{2}A_{T} <\infty,
% \end{align*}
% by Novikov condition,
% \begin{align*}
% \scE\brak{\int_{0}^{\cdot}G^{*}_{s}(m_{s})^{-1}dM_{s}} 
% \end{align*}
% is a uniformly integrable martingale. This implies that $\tilde M$ is
% $\tilde\bbP$- martingale since $|[\tilde M,\tilde M]_{T}|=|[M,M]_{T}|<\infty$.
% On the other hand, we have
% \begin{align*}
% f(s,Y_{s},Z_{s}m_{s})-f(s,\bar  Y_{s},\bar  Z_{s}m_{s})=\delta f_{s}+F_{s}\delta
%   Y_{s}+\delta Z_{s}m_{s}G_{s}
% \end{align*}
% and therefore, if we apply {\Ito} formula for
% $e^{\int_{0}^{t}F_{s}dA_{s}}\delta Y_{t}$,
% \begin{align*}
%  e^{\int_{0}^{t}F_{s}dA_{s}}\delta Y_{t}&=e^{\int_{0}^{T}F_{s}dA_{s}}\delta\xi+\int_{t}^{T}e^{\int_{0}^{s}F_{u}dA_{u}}\brak{\delta
%   f_{s}+\delta
%   Z_{s}m_{s}G_{s}}dA_{s}-\int_{t}^{T}e^{\int_{0}^{s}F_{u}dA_{u}}\delta
%   Z_{s}dM_{s}-\int_{t}^{T}e^{\int_{0}^{s}F_{u}dA_{u}} d(\delta
%   N)_{s}\\
% &=e^{\int_{0}^{T}F_{s}dA_{s}}\delta\xi+\int_{t}^{T}e^{\int_{0}^{s}F_{u}dA_{u}}\delta
%   f_{s}dA_{s}-\int_{t}^{T}e^{\int_{0}^{s}F_{u}dA_{u}}\delta
%   Z_{s}d\tilde M_{s}-\int_{t}^{T}e^{\int_{0}^{s}F_{u}dA_{u}} d(\delta
%   N)_{s}
% \end{align*}
% Since
% \begin{align*}
% \int_{0}^{t}e^{\int_{0}^{s}F_{u}dA_{u}}\delta
%   Z_{s}d\tilde M_{s}+\int_{0}^{t}e^{\int_{0}^{s}F_{u}dA_{u}} d(\delta
%   N)_{s}
% \end{align*}
% is a $\tilde\bbP$-martingale, 
If we take $\bbE_t$
on both side of the backward version of \eqref{lbsde}, we get
\begin{align*}
\delta Y_{t}=\frac{1}{\Gamma_{t}}\bbE_t\edg{\Gamma_{T}\delta\xi+\int_{t}^{T}\Gamma_{s}\delta
  f_{s}dA_{s}} \geq 0.
\end{align*}
%Therefore, the claim is proved.
\end{proof}
Now let us give the existence and uniqueness result when the terminal condition and the driver
depends on the path of $M'$ or $M$. Consider the following assumptions:
\begin{itemize}
\item[(S)] For any $\gamma\in D$ with $\norm{\gamma}_{\infty}\leq 1$,\[
\xi(M_{[0,T]}+\gamma),\xi(M'_{[0,T]})\in\bbL^{2}\quad\text{and}\quad
\bbE\int_{0}^{T}|f(s,M_{[0,s]}+\gamma,0,0)|^{2}dA_{s}, \bbE\int_{0}^{T}|f(s,M'_{[0,s]},0,0)|^{2}dA_{s}<\infty.
%\bbE \abs{\xi(M_{[0,T]})}^{2}+\bbE\int_{0}^{T} \abs{
%  f(t,M_{[0,t]},0,0)}^{2}dA_{t}<\infty.
\]
\item[(Lip)] There are nonnegative constants $C_{y}$ and $C_{z}$ such that
\begin{align*}
|f(t,\gamma_{[0,t]}, y,z)-f(t,\gamma_{[0,t]},y',z')|\leq C_{y}|y-y'|+C_{z}|z-z'|
\end{align*}
for all $\gamma\in D, t\in[0,T], y,y'\in\bbR^{d}, z,z'\in\bbR^{d\times
  n}$.
\end{itemize}
\begin{theorem}\label{prelim} Assume {\rm (S)} and {\rm (Lip)}.
The following BSDE
\begin{align}\label{cadlag}
Y_{t}=\xi(M'_{[0,T]})+\int_{t}^{T}f(s,M'_{[0,s]}, Y_{s-},Z_{s}m_{s})dA_{s}-\int_{t}^{T}Z_{s}dM_{s}-N_{T}+N_{t}, 
\end{align}
has a unique solution
$(Y,Z,N)\in\bbH^{2}\times\bbH_{m}^{2}\times\bbM^{2}$. Moreover,
$Y\in\bbS^{2}$ and there are path functionals
$\scY:[0,T]\times D\to\bbR^{d}, \scZ:[0,T]\times D\to\bbR^{d\times
  n}$, and $\scN:[0,T]\times D\to\bbR^{d}$ such that $\scY(t,\cdot),
\scZ(t,\cdot),\scN(t,\cdot)$ are $\scD$-measurable and
\begin{align*}
Y_{t}=\scY(t,M'_{[0,t]}),\qquad Z_{t}=\scZ(t,M'_{[0,t]}),\qquad\text{and }N_{t}=\scN(t,M'_{[0,t]})  
\end{align*}
 for each $t\in[0,T]$.
\end{theorem}
\begin{proof} 
Since $\xi(M'_{[0,T]})$ and $f(s, M'_{[0,s]},y,z)$ satisfies (STD), the
existence and uniqueness of solution follows from Proposition \ref{stdbsde}.
Consider a BSDE with the same terminal condition and driver under the
filtration $\bbF^{M'}$ and note that $M, m, A$ are adapted to $\bbF^{M'}$. By the same logic above, there
exists a unique solution which is adapted to $\bbF^{M'}$. Moreover,
since $\bbF^{M'}\subset\bbF$, this solution and the solution $(Y,Z,N)$
under original filtration should be the same. Therefore, $(Y,Z,N)$ should
be $\bbF^{M'}$-adapted. The existence of path functionals $\scY,\scZ$,
and $\scN$ is a result of Lemma \ref{msble}.
\end{proof}

\begin{theorem}\label{prelim2} Assume {\rm (S)} and {\rm (Lip)}.
Let $h\in[0,1/2]$ and
\begin{align*}
\xi^{h}(\gamma_{[0,T]})&:=\xi(\gamma_{[0,T]}+h {\bf e_{i}}^{*}
  \indicator{[u,T]})\\
f^{h}(s,\gamma_{[0,s]},y,z)&:=f(s,(\gamma+h {\bf e_{i}}^{*}
  \indicator{[u,T]})_{[0,s]},y,z). 
\end{align*}
Then, BSDE($\xi^{h},f^{h}$)
has a unique solution
$(Y^{h},Z^{h},N^{h})\in\bbH^{2}\times\bbH_{m}^{2}\times\bbM^{2}$. Moreover,
$Y^{h}\in\bbS^{2}$. In addition, for the same path functionals
$\scY,\scZ,\scN$ defined in
Theorem \ref{prelim}, we have
\begin{align*}
Y^{h}_{t}=\scY(t,M_{[0,t]}+h {\bf e_{i}}^{*}
  \indicator{[u,T]}),\qquad Z^{h}_{t}=\scZ(t,M_{[0,t]}+h {\bf e_{i}}^{*}
  \indicator{[u,T]}),\qquad\text{and }N^{h}_{t}=\scN(t,M_{[0,t]}+h {\bf e_{i}}^{*}
  \indicator{[u,T]})  
\end{align*}
 for each $t\in[0,T]$. In particular, this holds when $h=0$.
\end{theorem}
\begin{proof}
The existence of unique solution $(Y^{h},Z^{h},N^{h})$ and $Y^{h}$ being in $\bbS^{2}$
comes from Proposition \ref{stdbsde}. Let us denote $M^{h}_{t}:=M_t+h {\bf e_{i}}^{*}
  \indicator{[u,T]}(t)$. Let $(Y',Z',N')$ be the unique solution of \eqref{cadlag}
and $(\scY, \scZ,\scN)$ be the corresponding path functionals. Let us define $\Omega'\subset\Omega$ so that
$\bbP(\Omega')=1$ and
$Y'_{t}(\omega')=\scY(t,M'_{[0,t]}(\omega')),Z'_{t}(\omega')=\scZ(t,M'_{[0,t]}(\omega'))$, and
$N'_{t}(\omega')=\scN(t,M'_{[0,t]}(\omega'))$ for all
$\omega'\in\Omega'$. Note that, by our assumption on $\hat M$ and $M$, for all
$\omega'\in\Omega'$, there exists $\omega\in\Omega$ such
that $M(\omega)=M(\omega')$ and $\hat M(\omega)=h {\bf
  e_{i}}^{*}\indicator{[u,T]}$. For such $\omega$,\begin{align*}
Y'_{t}(\omega)=\scY(t,M^{h}_{[0,t]}(\omega)),\qquad Z'_{t}(\omega)=\scZ(t,M^{h}_{[0,t]}(\omega)), \qquad\text{ and } N'_{t}(\omega)=\scN(t,M^{h}_{[0,t]}(\omega)). 
\end{align*} and
\[
Y'_{t}(\omega)=\xi(M^{h}_{[0,T]}(\omega))+\int_{t}^{T}f(s,M^{h}_{[0,s]}(\omega), Y'_{s-}(\omega),Z'_{s}(\omega) m_{s}(\omega))dA_{s}(\omega)-\brak{\int_{t}^{T}Z'_{s}dM_{s}} (\omega)-N'_{T}(\omega)+N'_{t}(\omega).
\]
Since this holds for all $\omega$ realizing all possible paths of $M^{h}$ and $M$, the triplet \[(Y^{h}_{t}, Z^{h}_{t}, N^{h}_{t}):=(\scY(t,M^{h}_{[0,t]}),\scZ(t,M^{h}_{[0,t]}),\scN(t,M^{h}_{[0,t]})) \] is the unique solution of BSDE$(\xi^{h},f^{h})$.
\end{proof}
\begin{Remark}\label{jump}
	One may ask whether we can consider $M$ with jumps. If $M$ is a martingale with jumps, we know that the solution $(Y,Z,N)$ is adapted to the filtration generated by both $M$ and $M'$. However, $M$ is not adapted to $\bbF^{M'}$ and it is not obvious that whether the solution $(Y,Z,N)$ is actually adapted to the filtration generated only by $M'$.
\end{Remark}
\section{Path-differentiability of BSDE with Lipschitz Driver}\label{sec:diffbsde}
% In this section we assume {\rm (S), (Lip)}. Then, BSDE($\xi, f$) has a unique solution
% $(Y,Z,N)\in\bbH^{2}\times\bbH_{m}^{2}\times\bbM^{2}$ by Theorem \ref{prelim}.
Assuming that every martingale can be represented by stochastic integral with respect to $M$, $Z$ is often a ``path-derivative" of $Y$ in some sense: see, for example, El Karoui (1997, \cite{ElKaroui:1997dn}) for Malliavin calculus sense or Cont (2016, \cite{Bally:2016dw}) for functional {\Ito} calculus sense. This property is also called \textit{delta-hedging formula} due to its relationship with finance. We will prove this property and use it to study locally Lipschitz BSDEs. More precisely, we will find the almost sure uniform bound of the path-derivative of $Y$  to conclude $Z$ is uniformly bounded. Then, our locally Lipschitz BSDE becomes essentially Lipschitz BSDE and existence, uniqueness, and stability automatically follows from the results of Section \ref{sec:basic}.

However, both Malliavin calculus and functional {\Ito} calculus are not suitable for our problem as we described in the introduction. 
%Malliavin calculus is not suitable because we are studying general continuous martingale $M$. For example, if $M$ is Brownian motion stopped by hitting time, it is generally not Malliavin differentiable. Functional {\Ito} calculus is not appropriate either because it is not obvious what BSDE the path-derivative of solution satisfy and consequently, it is not easy to find the global bound of $Z$.
 Therefore, we will first define a new sense of path-derivative. Then, we will prove the delta-hedging formula for BSDE and show that the path-derivative of $(Y,Z,N)$ is the solution of the differentiated BSDE. This result, combined with Proposition \ref{bddY}, will be used to find the bound of $Z$ in Section \ref{sec:ext}.

\begin{definition}
For a random variable $V=\scV(M_{[0,T]})$ and
a vector ${\bf e}\in\bbR^{1\times n}$, we say $V$ is $\nabla^{\bf
  e}$-differentiable at $u$ if
\begin{align*}
\lim_{h\to0,\bbL^{2}}\frac{\scV(M_{[0,T]}+h {\bf e}^{*} \indicator{[u,T]})-\scV(M_{[0,T]})}{h}
\end{align*}
exists and we denote it by $\nabla^{\bf e}_{u}V$.
For a stochastic process $X=\scX(\cdot, M_{[0,\cdot]})$  and
a vector ${\bf e}\in\bbR^{1\times n}$, we say $X$ is $\nabla^{\bf
  e}$-, $\nabla^{{\bf e},m}$-, and $\nabla^{{\bf e},N}$-differentiable at
$u$ and define the
\emph{$\nabla^{\bf e}$-, $\nabla^{{\bf e},m}$-, and $\nabla^{{\bf e},N}$-derivative
at $u$} by
\begin{align*}
\nabla^{{\bf e}}_{u} X&:=\lim_{h\to0,\bbH^{2}}\frac{\scX(\cdot,(M_{[0,T]}+h
  {\bf e}^{*} \indicator{[u,T]})_{[0,\cdot]})-\scX(\cdot, M_{[0,\cdot]})}{h}\qquad\qquad\text{
  if }X\in\bbH^{2}\\
\nabla^{{\bf e},m}_{u} X&:=\lim_{h\to0,\bbH_{m}^{2}}\frac{\scX(\cdot,(M_{[0,T]}+h
  {\bf e}^{*} \indicator{[u,T]})_{[0,\cdot]})-\scX(\cdot, M_{[0,\cdot]})}{h}\qquad\qquad\text{
  if }X\in\bbH^{2}_{m}\\
\nabla^{{\bf e},N}_{u} X&:=\lim_{h\to0,\bbM^{2}}\frac{\scX(\cdot,(M_{[0,T]}+h
  {\bf e}^{*} \indicator{[u,T]})_{[0,\cdot]})-\scX(\cdot, M_{[0,\cdot]})}{h}\qquad\qquad\text{
  if }X\in\bbM^{2}
\end{align*}
if the corresponding limit exists.
In general, we denote
\begin{align*}
 \nabla_{u} V=\sum_{i=1}^{n}(\nabla^{{\bf e}_{i}}_{u} \scV){\bf e}_{i},\qquad\qquad
 \nabla^{m}_{u} X=\sum_{i=1}^{n}(\nabla^{{\bf e}_{i},m}_{u} \scX){\bf
                   e}_{i},\qquad\qquad\text{and }\quad
 \nabla^{N}_{u} X=\sum_{i=1}^{n}(\nabla^{{\bf e}_{i},N}_{u} \scX){\bf e}_{i}.
\end{align*}
where $\crl{{\bf e}_{i}}_{i=1,2,\cdots, n}$ is the standard basis of
$\bbR^{1\times n}$. If a random variable or a stochastic process
$\nabla^{\bf e}$-, $\nabla^{{\bf e},m}$-, and $\nabla^{{\bf
  e},N}$-differentiable for every ${\bf
  e}\in\bbR^{1\times n}$ and at almost every
$u\in[0,T]$, then we say the random variable or stochastic process is
differentiable with respect to $M$, or $\nabla$-, $\nabla^{m}$, $\nabla^{N}$-differentiable.
\end{definition}
\begin{remark}
This is a modified version of vertical functional {\Ito}
	derivative: see Dupire (2009, \cite{Dupire:2009eu}), Cont and
	Fournie (2013, \cite{Cont:2013hy}). The key differences are that it is time-parametrized and the
	convergence is in $L^{2}$-sense with respect to an appropriate
	measure.
\end{remark}
Note that the above definition crucially depends on the representation
path-functional $\scV$ or $\scX$. For example, let $c:\gamma\in D\mapsto c(\gamma)\in C([0,T]:\bbR^{n})$ be a function that removes
any jump part of $\gamma$. Then, the random variable
$V=\scV(M_{[0,T]})$ can also be written as $(\scV\circ
c)(M_{[0,T]})$. Note that $V\circ c$ is always $\nabla$-differentiable
with derivative $0$. 

Therefore, in order to incorporate above definitions to establish meaningful result in BSDE, we need to choose path representation carefully. Since (stochastic) integral are defined as a limit of time-partitioned sum, we have already restricted our representation of $\xi$ and $f$ by writing a BSDE. Therefore, one should choose the path-functional representation of $\xi$ and $f$ as limits of path-dependent functionals which depends on finite number of time sections of the path of $M$. Otherwise, the path-derivative of left hand side of BSDE may be different from path-derivative of right hand side which is written by stochastic integrals.
Let
$\scC_{k}$ be the set of continuously
differentiable functions from $(\bbR^{n})^{k}$ to $\bbR^{d}$. Let 
\begin{align*}
\scS:=\crl{H:D\to\bbR^{d}:\exists k\in\bbN, g\in \scC_{k}, 0=t^{k}_{0}\leq
  t^{k}_{1}\leq\cdots\leq t^{k}_{k}=T\;\; s.t.\;\;H(\gamma)=g(\gamma_{t^{k}_{1}}-\gamma_{t^{k}_{0}},\cdots, \gamma_{t^{k}_{k}}-\gamma_{t^{k}_{k-1}})} 
\end{align*}
For $\xi^{(k)}\in\scS$, we select $g\in\scC_{k}$ such that
$\xi^{(k)}(\gamma)=g(\gamma_{t^{k}_{1}}-\gamma_{t^{k}_{0}},\cdots,
\gamma_{t^{k}_{k}}-\gamma_{t^{k}_{k-1}})$ and then, we have
\begin{align*}
 \nabla^{\bf
  e}_{u}\xi^{(k)}(M_{[0,T]})=(\partial_{i}g)(M_{t^{k}_{1}}-M_{t^{k}_{0}},\cdots,
  M_{t^{k}_{k}}-M_{t^{k}_{k-1}}){\bf e}^{*}
\end{align*}
where $i$ satisfies $t^{k}_{i-1}< u\leq t^{k}_{i}$. For the
driver $f$ which depends on finitely many values
$(\gamma_{t^{k}_{i}})_{i=1,...,k}$, we choose its path functional and define the derivative
similarly: for each $(s,y,z)\in[0,T]\times\bbR^{d}\times \bbR^{d\times
  n}$, we treat $f(s,\cdot,y,z)$ as in $\scS$. For solution $(Y,Z,N)$
of BSDE($\xi,f$), we will always refer to the path functional
$\scY,\scZ,\scN$ defined in Theorem \ref{prelim2}. Note that the choice
of such $(\scY,\scZ,\scN)$ may not be unique but their derivatives are
unique stochastic processes which together forms a solution of
differentiated BSDE as we will see soon. This is consistent with the
result of Cont and Fournie (2013, \cite{Cont:2013hy}).

We would like to emphasize that these definitions are only needed in order to estimate the bound of $Z$ process by delta-hedging formula $Z_t=\nabla_t Y_t$ which is the next section's main result Theorem \ref{thm:ZdY}.
The key idea is the following:
\begin{itemize}
	\item[(i)]  Proposition \ref{prop:diff}: Consider BSDE$(\xi^{(k)},f^{(k)})$ where $(\xi^{(k)},f^{(k)})$ converges to $(\xi,f)$. Choose a representation of $(Y^{(k)},Z^{(k)},N^{(k)})$ of BSDE$(\xi^{(k)},f^{(k)})$ by Theorem \ref{prelim2} and establish BSDE satisfied by $(\nabla_tY^{(k)},\nabla^m_tZ^{(k)},\nabla^N_tN^{(k)})$.
	\item[(ii)] Theorem \ref{thm:ZdY0} and Corollary \ref{cor:ZdY0}: Prove $Z^{(k)}_t=\nabla_tY^{(k)}_t$ when $M$ has martingale representation property. 
	\item[(iii)] Theorem \ref{thm:ZdY}: Using Proposition \ref{bddY}, Proposition \ref{prop:diff}, and the fact that $Z^{(k)}$ converges to $Z$, find the bound of $Z$.
\end{itemize}

\begin{proposition}\label{prop:diff} Assume $\xi$ and $f$ satisfy {\rm (S), (Lip)}
	\begin{itemize}
		\item[{\rm (Diff)}] For each
		$(s,y,z)\in[0,T]\times\bbR^{d}\times\bbR^{d\times n}$, $\xi\in\scS$
		and $f(s,\cdot,y,z)\in\scS$. In addition, for all ${\bf
			e}\in\bbR^{1\times n}$ and almost every $u\in[0,T]$,
		$\nabla^{\bf e}_{u}\xi(M_{[0,T]})\in\bbL^{2}$ and $\nabla^{\bf
			e}_{u}f(\cdot, M_{[0,\cdot]},y',z')\;\vline_{(y',z')=(Y_{\cdot},(Zm)_{\cdot})}\in\bbH^{2}$,
		\item[{\rm (D)}] For all $t\in[0,T],\gamma\in D$,
		$f(t,\gamma_{[0,t]},y,z)$ is continuously differentiable with
		respect to $y$ and $z$. 
	\end{itemize}
	 and let $(Y,Z,N)$ to be the solution of {\rm
    BSDE($\xi,f$)}. Let $\scY,\scZ,\scN$ be the corresponding path
  functional as in Theorem \ref{prelim2}.
Then, the solution $Y_\cdot=\scY(\cdot,M_{[0,\cdot]}),Z_\cdot=\scZ(\cdot,M_{[0,\cdot]})$, and $N_\cdot=\scN(\cdot,M_{[0,\cdot]})$
are $\nabla$-, $\nabla^{m}$-, and $\nabla^{N}$-differentiable, respectively. Moreover, for each
$i=1,2,\cdots, n$
and almost every $u\in[0,T]$,
\begin{align*}
 ((\nabla^{\bf e_{i}}_{u} Y)_{t},(\nabla^{{\bf e_{i}},m}_{u}
  Z)_{t},(\nabla^{{\bf e_{i}},m}_{u}N)_t)=\left\{\begin{array}{ll}(0,0,0)&\text{if }u> t\\
                   (U_{t},V_{t},W_{t})&\text{if } u\leq t\end{array}
                                        \right. \qquad dt\otimes
                                        d\bbP\text{-almost everywhere}
\end{align*}
where
$(U,V,N)\in\bbH^{2}\times\bbH^{2}_{m}\times\bbM^{2}$
is the unique solution of the \eqref{stdBSDEeq} with the terminal
condition $\nabla^{\bf e_{i}}_{u}\xi(M_{[0,T]})$ and the driver
\begin{align*}
g(t,y,z)=\zeta_{t}+\eta_{t}y+\theta_{t}\cdot z
\end{align*}
Here, we defined
\begin{align*}
\zeta_{t}&:= (\nabla^{\bf e_{i}}_{u}f)(t,
  M_{[0,t]},y',z')\vline_{(y',z')=(Y_{t},Z_{t}m_{t})}\\
\eta_{t}&:=(\partial_{y}f)(t,  M_{[0,t]},Y_{t-},Z_{t}m_{t})\\
\theta_{t}&:=(\partial_{z}f)(t,
  M_{[0,t]},Y_{t-},Z_{t}m_{t})
\end{align*}
and
\[
\theta_{t}\cdot z:=\sum_{i,j}(\partial_{z^{ij}}f(t,
  M_{[0,t]},Y_{t},Z_{t}m_{t})z^{ij}
\]
\end{proposition}
 \begin{proof}
Note that $Y_{t}=\scY(t,M_{[0,t]}),Z_{t}=\scZ(t,M_{[0,t]})$, and
$N_{t}=\scN(t,M_{[0,t]})$, and therefore, if $u> t$, then $((\nabla^{\bf e_{i}}_{u} Y)_{t},(\nabla^{{\bf e_{i}},m}_{u}
  Z)_{t},\nabla^{{\bf e_{i}},m}_{u}N)=(0,0,0)$ because $(Y_{t},Z_{t},N_{t})$ is
  unaffected by the perturbation of $M$ at $u.$

Let us denote $M^{h}_{t}:=M_t+h {\bf e_{i}}^{*}
  \indicator{[u,T]}(t)$ and
\begin{align*}
\xi^{h}(\gamma_{[0,T]})&:=\xi(\gamma_{[0,T]}+h {\bf e_{i}}^{*}
  \indicator{[u,T]})\\
f^{h}(s,\gamma_{[0,s]},y,z)&:=f(s,(\gamma+h {\bf e_{i}}^{*}
  \indicator{[u,T]})_{[0,s]},y,z).  
\end{align*}
Note that $(\scY(t,M^{h}_{[0,t]}),\scZ(t,M^{h}_{[0,t]}),
\scN(t,M^{h}_{[0,t]}))$ is the unique solution of
BSDE($\xi^{h},f^{h}$) by Theorem \ref{prelim2}.
Let us define, for $u\leq t$,
\begin{align*}
\Xi^{h,u,i}&=\frac{\xi(M^{h}_{[0,T]})-\xi(M_{[0,T]})}{h}\\
U^{h,u,i}_{t}&=\frac{\scY(t,M^{h}_{[0,t]})-\scY(t,M_{[0,t]})}{h}\\
V^{h,u,i}_{t}&=\frac{\scZ(t,M^{h}_{[0,t]})-\scZ(t,M_{[0,T]})}{h}\\
W^{h,u,i}_{t}&:=\frac{\scN(t,M^{h}_{[0,t]})-\scN(t,M_{[0,t]})}{h}
\end{align*}
Then, for $t\geq u$, we have
\begin{align*}
U^{h,u,i}_{t}=\Xi^{h,u,i}+\int_{t}^{T}(\delta^{h,u,i}
  f)(s,M_{[0,s]},U^{h,u,i}_{s-}, V^{h,u,i}_{s}m_{s}) dA_{s}-\int_{t}^{T}V^{h,u,i}_{s}dM_{s}-W^{h,u,i}_{T}+W^{h,u,i}_{t}
\end{align*}
Here, we defined
\begin{align*}
&(\delta^{h,u,i} f)(t,M_{[0,t]},y,z)\\
&=\frac{1}{h}\edg{f^{h}(t,M_{[0,t]},\scY(t-,M_{[0,t-]})+hy,\scZ(t,M_{[0,t]})m_{t}+hz)-f(t,M_{[0,t]},\scY(t-,M_{[0,t-]}),\scZ(t,M_{[0,t]})m_{t})}\\
&= \zeta^{h,u,i}_{t}+\eta^{h,u,i}_{t}y+\theta^{h,u,i}_{t}\cdot z
\end{align*}
where
\begin{align*}
\zeta^{h,u,i}_{t}&:=
                   \frac{1}{h}\edg{f^{h}(t,M_{[0,t]},\scY(t-,M_{[0,t-]}),\scZ(t,M_{[0,t]})m_{t})-f(t,M_{[0,t]},\scY(t-,M_{[0,t-
]}),\scZ(t,M_{[0,t]}) m_{t})}\\
\eta^{h,u,i}_{t}y+\theta^{h,u,i}_{t}\cdot
  z&:=\frac{1}{h}[f^{h}(t,M_{[0,t]},\scY(t-,M_{[0,t-]})+hy,\scZ(t,M_{[0,t]})
     m_{t}+hz) \\ 
&\qquad\qquad \qquad\qquad \qquad\qquad \qquad\qquad
  -f^{h}(t,M_{[0,t]},\scY(t-,M_{[0,t-]}),\scZ(t,M_{[0,t]}) m_{t})].
\end{align*}
This BSDE has a unique solution $(U^{h,u,i}, V^{h,u,i},W^{h,u,i})\in\bbH^{2}\times\bbH^{2}_{m}\times\bbM^{2}$ because
$\Xi^{h,u,i}$ and $\delta^{h,u,i} f$ satisfies
(STD). In particular, 
$|\eta^{h,u,i}_{t}|\leq C_{y}$ and $|\theta^{h,u,i}_{t}|\leq C_{z}$ $dt\otimes
  d\bbP$-a.s. uniformly for all $h$ and $u$.
Also note that
\[
\begin{array}{ll}
\lim_{h\to0, \bbL^{2}} \Xi^{h,u,i}=\nabla^{\bf
                                    e_{i}}_{u}\xi(M_{[0,T]}), &\phantom{and}\qquad\lim_{h\to0, \bbH^{2}}\zeta^{h,u,i}=\zeta,\\
\lim_{h\to0}\eta^{h,u,i}_{t}=\eta_{t}, &\text{and}\qquad\lim_{h\to0}\theta^{h,u,i}_{t}=\theta_{t}
\end{array}\]
with $|\eta_{t}|\leq C_{y}$ and $|\theta_{t}|\leq C_{z}$. Then, by Proposition \ref{prop:stab},
\begin{align*}
\norm{U^{h,u,i}_{t}-U_{t}}^{2}_{2}&+\norm{U^{h,u,i}-U}^{2}_{\bbH^{2}}+\norm{V^{h,u,i}-V}^{2}_{\bbH_{m}^{2}}+\norm{W^{h,u,i}-W}_{\bbM^{2}}^{2}\\
&\leq
  2e^{K(2C_{y}+2C_{z}^{2}+2)}\brak{\norm{\Xi^{h,u,i}-\Xi}^{2}_{2}+\norm{\zeta^{h,u,i}-\zeta+(\eta^{h,u,i}-\eta)U+(\theta^{h,u,i}-\theta)\cdot
  V
  m}^{2}_{\bbH^{2}}}.
\end{align*}
By dominated convergence theorem, as $h\to 0$, we have
\begin{align*}
\norm{(\eta^{h,u,i}-\eta)U}_{\bbH^{2}}\to 0\qquad\text{and}\qquad
\norm{(\theta^{h,u,i}-\theta)\cdot Vm}_{\bbH^{2}} \to 0.
\end{align*}
Therefore,
\begin{align*}
U^{h,u,i}_{t}&\xrightarrow{ \bbL^{2}} U_{t}=\nabla^{\bf
               e_{i}}_{u}Y_{t}\qquad\text{for all }t\in[0,T]\\
U^{h,u,i}&\xrightarrow{\bbH^{2}} U=\nabla^{\bf e_{i}}_{u}Y\\
V^{h,u,i}&\xrightarrow{\bbH^{2}_{m}} V=\nabla^{{\bf e_{i}},m}_{u}Z\\
W^{h,u,i}&\xrightarrow{\bbM^{2}} W=\nabla^{{\bf e_{i}},N}_{u}N
\end{align*}
This implies $Y,Z$, and $N$
are $\nabla$-, $\nabla^{m}$-, and $\nabla^{N}$-differentiable, respectively, and
\[
(\nabla_{u}^{\bf e_{i}} Y,\nabla_{u}^{{\bf e_{i}},m}
Z,\nabla_{u}^{{\bf e_{i}},N} N)=(U,V,W)
\]
\end{proof}

It is widely known that the density process $Z$ can be thought as a ``derivative'' of $Y$
with respect to the driving martingale. Under the assumption that $M$
possesses martingale representation property, we can prove
this is indeed the case with our definition of path-derivative. To prove it, we need the following lemma.
\begin{lemma}
Consider $Z:=\scZ(\cdot,M_{\cdot})\in\bbH_{m}^{2}$ for
$\scZ:[0,T]\times D\to\bbR^{d\times n}$ such that $Z$ is
$\nabla^{{\bf e},m}_{u}$-differentiable. Then,
$\int_{t}^{T}Z_{s}dM_{s}$ is $\nabla_{u}^{\bf e}$-differentiable at
$u\in[0,T]$, and
\begin{align*}
\nabla^{\bf e}_{u} \int_{t}^{T}Z_{s}dM_{s}=\left\{\begin{array}{lll}&Z_{u}{\bf
  e}^{*}+\int_{u}^{T}(\nabla^{{\bf e},m}_{u}Z)_{s}dM_{s}&\qquad\text{
  for }u\in(t,T] \\
&&\\
&\int_{t}^{T}(\nabla^{{\bf
  e},m}_{u}Z)_{s}dM_{s}&\qquad\text{ for }u\in[0,t]\end{array}\right.
\end{align*}
\end{lemma}
\begin{proof}
Let us denote $M^{h,u}_{t}:=M_{t}+h{\bf e}^{*}1_{[u,T]}(t)$. Then
\begin{align*}
\int_{t}^{T}\scZ(s,M^{h,u}_{[0,s]})dM^{h,u}_{s}=\lim_{|\Pi|\to0}\sum_{i=0}^{N}\scZ(t_{i},M^{h,u}_{[0,t_{i}]})(M^{h,u}_{t_{i+1}}-M^{h,u}_{t_{i}})
\end{align*}
where $\Pi$ is a partition $\crl{0=t_{0}\leq t_{1}\leq\cdots\leq
  t_{N}=T}$ including a point $u\in[0,T]$ and $|\Pi|$ is the largest
interval of $\Pi$. Since the limit is
convergence in probability, we can take an appropriate subsequence of
$\Pi$ so that the convergence is almost sure sense. Likewise
\begin{align*}
\int_{t}^{T}\scZ(s,M_{[0,s]})dM_{s}=\lim_{|\Pi|\to0}\sum_{i=0}^{N}\scZ(t_{i},M_{[0,t_{i}]})(M_{t_{i+1}}-M_{t_{i}})
\end{align*}
using the same sequence of partition as above, by taking another
subsequence if necessary. Then we have
\begin{align*}
&\int_{t}^{T}\scZ(s,M^{h,u}_{[0,s]})dM^{h,u}_{s}-\int_{t}^{T}\scZ(s,M_{[0,s]})dM_{s}\\
&=\lim_{|\Pi|\to0}\sum_{i=0}^{N}\edg{\scZ(t_{i},M^{h,u}_{[0,t_{i}]})(M^{h,u}_{t_{i+1}}-M^{h,u}_{t_{i}})-\scZ(t_{i},M_{[0,t_{i}]})(M_{t_{i+1}}-M_{t_{i}})}\\
&=\lim_{|\Pi|\to0}\sum_{i=0}^{N}\edg{\brak{\scZ(t_{i},M^{h,u}_{[0,t_{i}]})-\scZ(t_{i},M_{[0,t_{i}]})}(M^{h,u}_{t_{i+1}}-M^{h,u}_{t_{i}})+\scZ(t_{i},M_{[0,t_{i}]})\brak{M^{h,u}_{t_{i+1}}-M^{h,u}_{t_{i}}-M_{t_{i+1}}+M_{t_{i}}}}\\
&=\int_{t\vee u}^{T}\brak{\scZ(s,M^{h,u}_{[0,s]})-\scZ(s,M_{[0,s]})} dM_{s}+hZ_{u}{\bf e}^{*}\indicator{u\in(t,T]}.
\end{align*}
Therefore,
\begin{align*}
\nabla^{\bf e}_{u} \int_{t}^{T}Z_{s}dM_{s}&=Z_{u}{\bf
  e}^{*}\indicator{\crl{u\in(t,T]}}+\lim_{h\to0,\bbL^{2}}
  \int_{t\vee u}^{T}\frac{\scZ(s,M^{h,u}_{[0,s]})-\scZ(,M_{[0,s]})}{h}
  dM_{s}\\
&=Z_{u}{\bf
  e}^{*}\indicator{\crl{u\in(t,T]}}+\int_{t\vee u}^{T}\edg{\lim_{h\to\infty,\bbH^{2}_{m}}\frac{\scZ(\cdot,M^{h,u}_{[0,\cdot]})-\scZ(\cdot,M_{[0,\cdot]})}{h}}_{s}
  dM_{s}\\
&=Z_{u}{\bf  e}^{*}\indicator{\crl{u\in(t,T]}}+\int_{t\vee u}^{T}(\nabla^{{\bf e},m}_{u}Z)_{s}dM_{s}
\end{align*}
\end{proof}

\begin{theorem}\label{thm:ZdY0} 
Assume that $\xi$ and $f$ satisfy {\rm (S),
    (Lip), (Diff),} and {\rm (D)} and let $(Y,Z,N)$ to be the solution
  of {\rm BSDE($\xi,f$)}. Let $\scY,\scZ,\scN$ be the corresponding path
  functional as in Theorem \ref{prelim2}.
Then, \[\nabla_{u}Y_{u}=Z_{u}+\nabla_uN_u,\qquad \text{$du\otimes
d\bbP$-almost everywhere.} \]
\end{theorem}
\begin{proof}
 Since $M$ has martingale representation property and $(Y,Z,N)$ is
 $\bbF^{M}$-adapted by Theorem \ref{prelim2}, we know $N\equiv0$ and $Y$ has continuous path. Therefore, we have the following forward SDE.
\begin{align*}
Y_{t}=Y_{0}-\int_{0}^{t} f(s,M_{[0,s]},Y_{s-},Z_{s}m_{s})dA_{s}+\int_{0}^{t}Z_{s}dM_{s}+N_t
\end{align*}
Let $u\in(0,T]$ and 
\begin{align*}
M_{s}^{h}&:=M_{s}+h{\bf e_{i}}^{*}1_{[u,T]}(s)\\
f^{h}_{s}&:=f(s,M^{h}_{[0,s]},\scY(s-,M^{h}_{[0,s-]}),\scZ(s,M^{h}_{[0,s]})m_{s})
%A_{s}^{h}&:= {\rm tr}[M^{h},M^{h}]_{s}=A_{s}+h^{2}1_{[u,T]}(s)
\end{align*}
where $\scY$ and $\scZ$ are defined as in the proof of Proposition \ref{prop:diff}.
Let us define $\zeta,\eta,\theta,\zeta^{h,u,i},\eta^{h,u,i},\theta^{h,u,i}$, and $(U^{h,u,i},V^{h,u,i}, W^{h,u,i})$ as in the proof of Proposition
\ref{prop:diff}.
 Note that $\zeta^{h,u,i}$ converges to $\zeta$ in $\bbH^2$;
$(U^{h,u,i}, V^{h,u,i}, W^{h,u,i})$ converge to $(\nabla_{u}^{\bf  e_{i}}Y,  \nabla_{u}^{{\bf e_{i}},m}Z,\nabla_{u}^{{\bf e_{i}},N}W)$ in $\bbH^{2}\times\bbH^2_m\times\bbM^2$;
$\eta^{h,u,i},\theta^{h,u,i}$ converge to $\eta,\theta$ in
$dt\otimes d\bbP$-a.e. sense; and
$\eta^{h,u,i},\theta^{h,u,i},\eta,\theta$ are bounded $dt\otimes
d\bbP$-a.e. sense.
Therefore,
\begin{align*}
&\bbE\abs{\int_{0}^{t}\brak{\frac{f_{s}^{h}-f_{s}^{0}}{h}-1_{[u,T]}(s)\edg{\zeta_{s}+\eta_{s}\nabla_{u}^{\bf
  e_{i}}Y_{s-}+\theta_{s}\cdot
  (\nabla_{u}^{{\bf e_{i}},m}Z)_{s}m_{s}}}dA_{s}}^{2}\\
& \leq K\bbE\int_{0}^{t}\abs{\frac{f_{s}^{h}-f_{s}^{0}}{h}-1_{[u,T]}(s)\edg{\zeta_{s}+\eta_{s}\nabla_{u}^{\bf
  e_{i}}Y_{s-}+\theta_{s}\cdot
  (\nabla_{u}^{{\bf e_{i}},m}Z)_{s}m_{s}}}^{2}dA_{s}\\
&\leq
  K\brak{\norm{\zeta^{h,u,i}-\zeta}_{\bbH^{2}}+\norm{\eta^{h,u,i}U^{h,u,i}-\eta
  \nabla_{u}^{\bf
  e_{i}}Y}_{\bbH^{2}}+\norm{\theta^{h,u,i}_{s}\cdot V^{h,u,i}_{s}m_{s}-\theta_{s}\cdot
  (\nabla_{u}^{{\bf e_{i}},m}Z)_{s}m_{s}}_{\bbH^{2}}}\\
&\leq K\big(\norm{\zeta^{h,u,i}-\zeta}_{\bbH^{2}}+\norm{\eta^{h,u,i}(U^{h,u,i}- \nabla_{u}^{\bf
  e_{i}}Y)}_{\bbH^{2}}+\norm{(\eta^{h,u,i}-\eta)
  \nabla_{u}^{\bf
  e_{i}}Y}_{\bbH^{2}}\\
&\qquad\qquad +\norm{\theta^{h,u,i}_{s}\cdot (V^{h,u,i}_{s}m_{s}-  (\nabla_{u}^{{\bf e_{i}},m}Z)_{s}m_{s})}_{\bbH^{2}}+\norm{(\theta^{h,u,i}_{s}-\theta_{s})\cdot
  (\nabla_{u}^{{\bf e_{i}},m}Z)_{s}m_{s}}_{\bbH^{2}}\big)
\xrightarrow{h\to0}0
\end{align*}
by dominated convergence theorem.
As a result, 
\begin{align*}
% \lim_{h\to0,\bbL^{2}}\frac{1}{h}\brak{\int_{0}^{t}f^{h}_{s}dA^{h}_{s}-\int_{0}^{t}f^{0}dA_{s}}=
  \lim_{h\to0,\bbL^{2}}\int_{0}^{t}\frac{f^{h}_{s}-f^{0}_{s}}{h}dA_{s}=\int_{u\wedge
  t}^{t}\brak{\zeta_{s}+\eta_{s}\nabla_{u}^{\bf
  e_{i}}Y_{s-}+\theta_{s}\cdot
  (\nabla_{u}^{{\bf e_{i}},m}Z)_{s}m_{s}} dA_{s}.
\end{align*}
Then, by our previous lemma, for $u\in (0,t]$
\begin{align*}
\nabla_{u}^{\bf e_{i}}Y_{t}=Z_{u}{\bf e_{i}}^{*}-\int_{u}^{t}\brak{\zeta_{s}+\eta_{s}\nabla_{u}^{\bf
  e_{i}}Y_{s-}+\theta_{s}\cdot
  (\nabla_{u}^{{\bf e_{i}},m}Z)_{s}m_{s}} dA_{s}+\int_{u}^{t}(\nabla^{{\bf e_{i}},m}_{u}Z)_{s}dM_{s}+ \nabla_{u}^{\bf e_{i}}N_{t}
\end{align*}
Since $\nabla_{u}Y_{t}$ and $\nabla_{u}N_t$ are right continuous, we prove the claim by letting $t\searrow u$.
\end{proof}
When $M$ has martingale representation property, then $N\equiv0$ and above theorem implies the following corollary which we will use in section 5 and 6.
\begin{corollary}\label{cor:ZdY0}
Assume the conditions in Theorem \ref{thm:ZdY0}. In addition, assume that 
\begin{itemize}
\item[{\rm (M)}] $M$ has martingale
representation property; that is, any $(\bbF^{M},\bbP)$ martingale $X$ such that $\bbE{\rm tr}[X,X]_T<\infty$ can be expressed as $X_t=X_0+\int_{0}^{t}Z_{s}dM_{s}$ for some $Z\in\bbH^{2}_{m}$.
\end{itemize}
Let $(Y,Z,N)$ to be the solution
of {\rm BSDE($\xi,f$)}. Let $\scY,\scZ,\scN$ be the corresponding path
functional as in Theorem \ref{prelim2}.
Then, \[\nabla_{u}Y_{u}=Z_{u},\qquad \text{$du\otimes
	d\bbP$-almost everywhere.} \]
\end{corollary}
\section{BSDE with Locally Lipschitz Driver}\label{sec:ext}
In this section, we always assume (M). This implies $Y$ has
continuous path, therefore, $Y_{s-}=Y_{s}$ for all $s$ and $N\equiv0$. Using Corollary \ref{cor:ZdY0}, this martingale representation property enables us to bound $Z$ process by bounding $\nabla$-derivative of $Y$.
\subsection{A Priori Estimate of $Z$}
Let $k$ be an positive integer and $0=t^{k}_{0}\leq
t^{k}_{1}\leq\cdots\leq t^{k}_{k}=T$ be a partition of $[0,T]$. For $\gamma\in D$, we define
\begin{align*}
P^{(k)}(\gamma)&:=(\gamma_{t^{k}_{1}}-\gamma_{t^{k}_{0}},\cdots,
  \gamma_{t^{k}_{k}}-\gamma_{t^{k}_{k-1}})\\
L^{(k)}(a_{1},...,a_{k})&:=\sum_{i=1}^{k}a_{i}1_{[t^{k}_{i},T]}.
%P^{(k)}&:=\sum_{i=0}^{k}(\gamma_{t^{k}_{i}})1_{[t^{k}_{i},t^{k}_{i+1})}.=L^{(k)}\circ P^{(k)}
\end{align*}
% Note that $P^{(k)}\circ L^{(k)}(a_{1},...,a_{k})=(a_{1},...,a_{k})$.
Let us denote $x \in
\bbR^{kn}$ and let $\varphi \in C^{\infty}_c
(\bbR^{kn};\bbR)$ be the 
mollifier
$$
\varphi(x) := 
\left\{
\begin{array}{ll}
\lambda \exp \brak{- \frac{1}{1- |x|^2}} &\text { if } |x|<1\\
0 &\text{ otherwise}
\end{array} \right.,
$$
where the constant $\lambda \in \mathbb{R}_+$ is chosen so that $\int_{\mathbb{R}^{kn}} \varphi(x) dx = 1$.
Set $\varphi^{(k)}(x) := k^{kn} \varphi(k x)$, and define
\begin{align*}
\xi^{(k)}&:=\edg{(\xi\circ L^{(k)})*\varphi^{(k)}}\circ P^{(k)}\\
f^{(k)}(s,\gamma,y,z)&:= \int_{\bbR^{kn}}f(s,L^{(k)}(P^{(k)}(\gamma)-x'),y,z)\varphi^{(k)}(x')dx'.
\end{align*}
\begin{theorem}\label{thm:ZdY}
Assume that $\xi$ and $f$ satisfy {\rm (M), (Lip)}. In addition, assume the following condition:
\begin{itemize}
\item[{\rm (Diff')}] 
\begin{itemize}
\item $\xi(0)<\infty$ and $\int_0^T |f(s,0,0,0)|^2dA_s<\infty$.
%\item $\norm{M_{[0,T]}}_{\infty}\in\bbL^{2}$ and
\item $
\norm{M_{[0,T]}-(L^{(k)}\circ
  P^{(k)})(M_{[0,T]})}_{\infty}\xrightarrow[k\to\infty]{\bbL^{2}} 0
$.
\item There are $D_{\xi},D_{f}\in\bbR_{+}$
  such that, for all $ \gamma,\gamma'\in D,
  (s,y,z)\in[0,T]\times\bbR^{d}\times\bbR^{d\times n}$,
\begin{align*}
|\xi(\gamma)-\xi(\gamma')|\leq
                            D_{\xi}\norm{\gamma-\gamma'}_{\infty}\quad\text{
                            and}\quad
|f(s,\gamma,y,z)-f(s,\gamma',y,z)|\leq D_{f}\norm{\gamma-\gamma'}_{\infty}.
\end{align*}
\end{itemize}
\end{itemize}
Then, BSDE($\xi,f$) has a unique solution
$(Y,Z,N)\in\bbH^{2}\times\bbH^{2}_{m}\times\bbM^{2}$. Moreover, $Y$
has continuous path,
$N\equiv 0$ and 
\[
|Z_{t}|\leq \sqrt{D_{\xi}^{2}+D_{f}^{2}K}e^{\half K(2C_{y}+C_{z}^{2}+1)}\qquad dt\otimes d\bbP\text{-a.e.}
\]
\end{theorem}
\begin{Remark}
It is easy to see that the second condition of (Diff') holds when $M$ is a Brownian motion and $t^{k}_{i}=iT/k$. More generally, let $W$ be Brownian motion and assume that $M$ has a martingale representation
\[
M_t=\int_0^t \eta_sdW_s
\]
where there exists a constant $C$ such that $|\eta_s|\leq C$ almost surely. Then
\begin{align*}
\bbE\norm{M_{[0,T]}-(L^{(k)}\circ
	P^{(k)})(M_{[0,T]})}_{\infty}^2&=\bbE\sup_{i=0,...,k-1}\sup_{t_i\leq t<t_{i+1}}\abs{\int_{t_i}^t\eta_sdW_s}^2\leq\bbE\edg{\sum_{i=0}^{k-1}\sup_{t_i\leq t<t_{i+1}}\abs{\int_{t_i}^t\eta_sdW_s}^4}^{1/2}\\
&\leq\edg{\sum_{i=0}^{k-1}\bbE\sup_{t_i\leq t<t_{i+1}}\abs{\int_{t_i}^t\eta_sdW_s}^4}^{1/2}\leq\edg{\sum_{i=0}^{k-1}\bar C\bbE\brak{\int_{t_i}^{t_{i+1}}|\eta_s|^2ds}^2}^{1/2}\\
&\leq \sqrt{\bar C}C^2\edg{\sum_{i=0}^{k-1}(T/k)^2}^{1/2}\leq \sqrt{\bar C}C^2T\frac{1}{\sqrt{k}}\xrightarrow{k\to\infty}0.
\end{align*}
The inequalities are based on
\[
\sup_{i=0,...,k-1}|a_i|\leq\sqrt{\sum_{i=0}^{k-1}|a_i|^2},
\]
Jensen inequality, and Burkholder-Davis-Gundy inequality. Here, we used $\bar C$ for the contant of Burkholder-Davis-Gundy inequality.
Therefore, such $M$ satisfies the second condition of (Diff').
\end{Remark}
Before we proceed to the proof, let us observe the following facts.
\begin{lemma}\label{lmn:ZdY} Under the assumption of Theorem \ref{thm:ZdY}, we have
  the following results.
\begin{itemize}
\item[(i)] $(\xi,f)$ satisfies {\rm (S)} and $(\xi^{(k)},f^{(k)})$ satisfies {\rm (S), (Diff), (Lip)}
\item[(ii)] $|\nabla_{u} \xi^{(k)}(M_{[0,T]})|\leq D_{\xi}$ and
$|\nabla_{u}f^{(k)}(t,M_{[0,t]},y,z)|\leq D_{f}$ for all $u\in[0,T],
(y,z)\in\bbR^{d}\times\bbR^{d\times n}$ in $dt\otimes d\bbP$-a.e. 
\item[(iii)] For solution $(Y,Z,N)\in\bbH^{2}\times\bbH^{2}_{m}\times\bbM^{2}$ of BSDE($\xi,f$), $Y$ has a
  continuous path, $N\equiv 0$, and
\begin{align*}
\xi^{(k)}(M_{[0,T]})
  \xrightarrow[k\to\infty]{\bbL^{2}}\xi(M_{[0,T]})\quad\text{and}\quad f^{(k)}(\cdot,M_{[0,\cdot]},Y_{\cdot-},(Zm)_{\cdot})\xrightarrow[k\to\infty]{\bbH^{2}}
f(\cdot, M_{[0,\cdot]},Y_{\cdot-},(Zm)_{\cdot}).
\end{align*}
\end{itemize}
\end{lemma}
\begin{proof} It is easy to see $(\xi^{(k)},f^{(k)})$ satisfies (Lip).
Note that \begin{align*}
\norm{\xi\brak{M_{[0,T]}+\gamma}}_{2}&\leq\norm{\xi(0)+D_{\xi}\norm{M_{[0,T]}}_{\infty}+\norm{\gamma}_{\infty}}_{2}<\infty\\
\norm{\xi(M'_{[0,T]})}_{2}&\leq
            \norm{\xi(0)+D_{\xi}\norm{M_{[0,T]}}_{\infty}+D_{\xi}\norm{\hat
            M_{[0,T]}}_{\infty}}_{2}<\infty
\end{align*}
since $\norm{M_{[0,T]}}_{\infty}\in\bbL^{2}$ by Burkholder-Davis-Gundy inequality and $\norm{\hat M_{[0,T]}}_{\infty}$ is bounded by the number of
jumps of $M'$ which is a Poisson random variable. Therefore, $\xi$ satisfies (S). We can use the same arguement to show that $f$
satisfies (S) as well.
Note that
\begin{align*}
|\xi^{(k)}(\gamma)-\xi^{(k)}(\gamma')|&\leq\int_{\bbR^{kn}}\abs{\xi(L^{(k)}(P^{(k)}(\gamma)-x'))-\xi(L^{(k)}(P^{(k)}(\gamma')-x'))}\varphi^{(k)}(x')dx'\\
&\leq
  D_{\xi}\int_{\bbR^{kn}}\norm{L^{(k)}(P^{(k)}(\gamma)-x')-L^{(k)}(P^{(k)}(\gamma')-x')}_{\infty}\varphi^{(k)}(x')dx'\\
&\leq D_{\xi}\norm{\gamma-\gamma'}_{\infty}.
\end{align*}
Likewise,
\begin{align*}
|f^{(k)}(s,\gamma,y,z)-f^{(k)}(s,\gamma',y,z)|&\leq D_{f}\norm{\gamma-\gamma'}_{\infty}.
\end{align*}
Using the same argument for $(\xi,f)$, this implies $(\xi^{(k)},
f^{(k)})$ satisfies (S). Moreover, it also implies (Diff) and (ii) because of
Lipschitzness and the convolution with the mollifier
$\varphi^{(k)}$.

Lastly,  since $M$ has martingale representation property and $(Y,Z,N)$ is
 $\bbF^{M}$-adapted by Theorem \ref{prelim2}, we know $N\equiv
 0$ and $Y$ has continuous path. Also, note that
\begin{align*}
&\norm{\xi(M_{[0,T]})-\xi^{(k)}(M_{[0,T]})}_{2}\\
&\leq
  \norm{\xi(M_{[0,T]})-(\xi\circ L^{(k)}\circ
  P^{(k)})(M_{[0,T]})}_{2}+\norm{\int_{\bbR^{kn}}\abs{(\xi\circ
  L^{(k)})(P^{(k)}(\gamma))-(\xi\circ
  L^{(k)})(P^{(k)}(\gamma)-x')}\varphi^{(k)}(x')dx'}_{2}\\
&\leq D_{\xi}\norm{\norm{M_{[0,T]}- (L^{(k)}\circ
  P^{(k)})(M_{[0,T]})}_{\infty}}_{2}+D_{\xi}\norm{\int_{\bbR^{kn}}\max_{i}|x'_{i}|\varphi^{(k)}(x')dx'}_{2}\xrightarrow{k\to\infty} 0
\end{align*}
We can argue similarly for $f^{(k)}$ to conclude (iii) holds.
\end{proof}
\begin{proof}[Proof of Theorem \ref{thm:ZdY}]
Let us denote $x = (y,z) \in
\bbR^{d}\times\bbR^{d\times n}$ and let $\beta \in C^{\infty}_c
(\bbR^{d}\times\bbR^{d\times n};\bbR)$ be the 
mollifier
$$
\beta(x) := 
\left\{
\begin{array}{ll}
\lambda \exp \brak{- \frac{1}{1- |x|^2}} &\text { if } |x|<1\\
0 &\text{ otherwise}
\end{array} \right.,
$$
where the constant $\lambda \in \mathbb{R}_+$ is chosen so that $\int_{\mathbb{R}^{dn+d}} \beta(x) dx = 1$.
Set $\beta^{(m)}(x) := m^{d+dn} \beta(m x)$, $m \in \mathbb{N} \setminus \crl{0}$, and define
\[
f^{(k,m)}(t,\gamma_{[0,t]},x):= \int_{\bbR^{dn+d}} f^{(k)}(t,\gamma_{[0,t]},x-x') \beta^{(m)}(x') dx'.
\]
Then, it is easy to check that (S), (Lip), (Diff) for $f^{(k)}$ implies that $(\xi^{(k)}, f^{(k,m)})$ satisfies (S),
(Lip), (Diff), and (D). Therefore, there
exists solution $(Y^{(k,m)},Z^{(k,m)})$ of the BSDE
\begin{align*}
Y^{(k,m)}_{t}=\xi^{(k)}(M_{[0,T]})+\int_{t}^{T}f^{(k,m)}(s,M_{[0,s]},Y^{(k,m)}_{s-},Z^{(k,m)}_{s}m_{s})dA_{s}-\int_{t}^{T}Z^{(k,m)}_{s}dM_{s}
\end{align*}
From our Proposition \ref{bddY} and Proposition \ref{prop:diff}, we know
\begin{align*}
|\nabla_{u}Y^{(k,m)}_{t}|\leq\sqrt{D_{\xi}^{2}+D_{f}^{2}K}e^{\half K(2C_{y}+C_{z}^{2}+1)}
\end{align*}
for all $u\in[0,T]$.
Then, from Corollary \ref{cor:ZdY0}, for all $k,m\in\bbN$,
\[
|Z^{(k,m)}_{t}|\leq\sqrt{D_{\xi}^{2}+D_{f}^{2}K}e^{\half K(2C_{y}+C_{z}^{2}+1)}
\]$dt\otimes d\bbP$ -almost everywhere. 
By Proposition \ref{prop:stab}, we have
\begin{align*}
&\norm{Y-Y^{(k,k)}}_{\bbH^{2}}^{2}+\norm{Z-Z^{(k,k)}}^{2}_{\bbH^{2}_{m}}\\
&\leq 2e^{2K(C_{y}+C_{z}^{2}+1)}\brak{\norm{\xi(M_{[0,T]})-\xi^{(k)}(M_{[0,T]})}^{2}_{2}+\norm{f(\cdot,
  M_{[0,\cdot]}, Y_{\cdot-},Z_{\cdot}m_{\cdot})-f^{(k,k)}(\cdot,
  M_{[0,\cdot]}, Y_{\cdot-},Z_{\cdot}m_{\cdot})}_{\bbH^{2}}^{2}}.
\end{align*}
Since $f^{(k)}$ is Lipschitz,
\begin{align*}
&\abs{ f^{(k)}(t,
  M_{[0,t]}, Y_{t-},Z_{t}m_{t})-f^{(k,k)}(t,
  M_{[0,t]}, Y_{t-},Z_{t}m_{t})}\\
&\leq\abs{\int_{\bbR^{dn+d}} \brak{f^{(k)}(t,M_{[0,t]}Y_{t-},Z_{t}m_{t})-f^{(k)}(t,M_{[0,t]},(Y_{t-},Z_{t}m_{t})-x')}
  \beta^{(k)}(x') dx'}\\
&\leq \brak{C_{y}+C_{z}}\int_{\bbR^{dn+d}}|x'|\beta^{(k)}(x')dx'\xrightarrow{m\to\infty}0.
\end{align*}
Combined with the previous lemma, this implies that 
\[
\norm{f(\cdot,
  M_{[0,\cdot]}, Y_{\cdot-},Z_{\cdot}m_{\cdot})-f^{(k,k)}(\cdot,
  M_{[0,\cdot]},
  Y_{\cdot-},Z_{\cdot}m_{\cdot})}_{\bbH^{2}}\xrightarrow{k\to\infty} 0.
\]
Therefore, $Y^{(k,k)}\to Y$ in $\bbH^{2}$ and $Z^{(k,k)}\to Z$ in $\bbH_{m}^{2}$ with
$|Z^{(k,k)}_{t}|\leq\sqrt{D_{\xi}^{2}+D_{f}^{2}K}e^{\half K(2C_{y}+C_{z}^{2}+1)}.$ Therefore,
\[|Z_{t}|\leq\sqrt{D_{\xi}^{2}+D_{f}^{2}K}e^{\half K(2C_{y}+C_{z}^{2}+1)}.\]
\end{proof}

\subsection{Existence and uniqueness when $d\geq 1$ and $[M,M]_{T}$ is
small.}
\begin{theorem}\label{localmulti}
Assume the following conditions: {\rm (M), (Diff'),} and
\begin{itemize}
\item[{\rm (Loc)}] There exists a nondecreasing function 
$\rho : \mathbb{R}_+ \to \mathbb{R}_+$ such that
\[
|f(t,\gamma_{[0,t]},y,z)-f(t,\gamma_{[0,t]},y',z')| \leq \rho(|z| \vee |z'|)\brak{|y-y'|+ |z-z'|}
\]
for all $t \in[0,T]$, $y,y' \in \mathbb{R}^{d}$ and $z,z' \in
\mathbb{R}^{d\times n}$.
%\item[{\rm (M')}] The assumption {\rm (M)} holds and moreover, there exists $K_{m}$
%  such that $|m_{t}|\leq K_{m}$ $dt\otimes d\bbP$-a.e.
\end{itemize}
Assume that $K$ is small enough so that there is $R\in\bbR_{+}$ satisfying
\[
\sqrt{D_{\xi}^{2}+D_{f}^{2}K}e^{\half K(\rho(R)+1)^2}\leq R
\]
Then, the BSDE
\[
Y_{t}=\xi(M_{[0,T]})+\int_{t}^{T}f(s,M_{[0,s]},Y_{s},Z_{s}m_{s})dA_{s}-\int_{t}^{T}Z_{s}dM_{s}
\]
has a unique solution $(Y,Z)\in\bbH^{2}\times \bbH_{m}^{2}$ such that $Z$
is bounded. Moreover, $|Z_{t}|\leq R$ in $dt\otimes d\bbP$-almost
everywhere sense.
\end{theorem}
\begin{proof}
For $R$ in the assumption, consider the BSDE with the terminal condition
$\xi(M_{[0,T]})$ and the driver
\[
g(t,\gamma_{[0,t]},y,z):=f\brak{t,\gamma_{[0,t]},y,\frac{Rz}{|z|\vee R}}.
\]
Then, $g$ satisfies (Diff') and (Lip) with Lipschitz coefficient of the
driver $\rho(R)=C_{y}=C_{z}$. Therefore, there
exists a unique solution $(U,V)\in\bbH^{2}\times\bbH_{m}^{2}$ for the following BSDE
\begin{align*}
U_{t}=\xi(M_{[0,T]})+\int_{t}^{T}g(s, M_{[0,s]},U_{s},V_{s}m_{s})dA_{s}-\int_{t}^{T}V_{s}dM_{s}
\end{align*}
and $V$ is bounded by
\[
|V_{t}|\leq \sqrt{D_{\xi}^{2}+D_{f}^{2}K}e^{\half K(\rho(R)+1)^2}.
\]
Therefore, since $|m_{t}|=1$ for all $t$ (see Section 2),
\[
|V_{t}m_{t}|\leq |V_{t}||m_{t}|\leq \sqrt{D_{\xi}^{2}+D_{f}^{2}K}e^{\half K(\rho(R)+1)^2}\leq R.
\]
This implies that $(U,V)$ is also a solution of 
\[
Y_{t}=\xi(M_{[0,T]})+\int_{t}^{T}f(s, M_{[0,s]},Y_{s},Z_{s}m_{s})dA_{s}-\int_{t}^{T}Z_{s}dM_{s}.
\]

Now let us show the uniqueness. Assume that $(Y',Z')$ is another
solution such that $Z'$ is bounded by $Q$. Without loss of generality,
we can assume $Q\geq R$. Then, if we consider
\[
h(t,\gamma_{[0,t]},y,z):=f\brak{t,\gamma_{[0,t]},y,\frac{Qz}{|z|\vee Q}},
\]
then BSDE($\xi,h$) has a unique solution in $\bbH^{2}\times\bbH_{m}^{2}$. Since $(Y,Z)$
and $(Y',Z')$ are both solution to such BSDE, we have $(Y,Z)=(Y',Z')$.
 \end{proof}
\subsection{Explosion of solution when $d>1$ and $[M,M]_{T}$ is large.}
\label{sec:counter}
If $d>1$, the result on the previous subsection cannot extend to arbitrary large $K$ in general.
This can be shown by the following counterexample which is inspired by Chang et al. (1992, \cite{Chang:1992p29958}).

For $\delta>0$, let $M_{t}:=\sqrt{2}(W_{t}^{\tau}-W^{\tau}_{T-\delta})\indicator{[T-\delta,\infty)}(t)$
where $W_{t}$ is a two dimensional Brownian motion and
\[
\tau:=\inf\crl{t\in[T-\delta,\infty):|W_{t}-W_{T-\delta}|\geq 1/\sqrt{2}}\wedge T.
\]
Note that $M$ satisfies (M) because of Lemma 2.1 of Peng (1991, \cite{Peng:1991ku}).
Define the terminal condition $\xi$ as
\begin{align*}
\xi(M_{T})&:=\begin{pmatrix}\cos\theta_{T}\sin g_{0}(R_{T})\\\sin\theta_{T}\sin
  g_{0} (R_{T})\\\cos g_{0} (R_{T})\end{pmatrix}
\end{align*}
where $g_{0}:\bbR_{+}\to\bbR$ is a smooth function with bounded
derivatives of all order
and  
$(R_{s},\theta_{s})$ is the polar coordinate of $M_{s}$. Note that
$\xi$ is a smooth function with bounded derivative.

We let, for $\veps\in(0,1)$, 
\begin{align*}
\phi(r):=\arccos\brak{\frac{\lambda^{2}-r^{2(1+\veps)}}{\lambda^{2}+r^{2(1+\veps)}}}
\end{align*}
where $\lambda$ is big enough so that
$\cos\phi(r)\geq(1+\veps)^{-1}$ for $r\in[0,1]$.
We choose a smooth function $g_0$ so that
\begin{align*}
g_{0}(r)\geq \arccos\brak{\frac{1-r^{2}}{1+r^{2}}} +\phi(r)
\end{align*}
for $r\in[0,1]$ and $g_0(0)=0$. Note that $g_0$ has bounded derivatives all orders on $[0,1]$.

Let us show the following BSDE driven by $M$ have a solution $(Y,Z)$
such that $Z$ is bounded only
when $[M,M]_{T}$ is small enough:
\begin{align}\label{counterex}
Y_{t}=\xi(M_{T})+\int_{t}^{T}\half|Z_{s}m_{s}|^{2}\frac{Y_{s}}{|Y_{s}|\vee 1}dA_{s}-\int_{t}^{T}Z_{s}dM_{s}.
\end{align}
Since we have
$
[M,M]_{t}:=2((s\wedge\tau-(T-\delta))\vee 0)\bbI,
$
we have $m_{s}=\frac{1}{\sqrt{2}}\bbI$ and
$A_{s}=4(s\wedge\tau-(T-\delta))\vee 0$. Note that
$\esssup_{\omega\in\Omega}A_{T}=4\delta$.
It is easy to check that (Loc) and (Diff')
if we let $\rho(x)=x+\half|x|^{2}$. Then, by Theorem \ref{localmulti}, \eqref{counterex} has a
unique solution $(Y,Z)\in\bbH^{2}\times\bbH_{m}^{2}$ such that $Z$ is bounded if $\delta$ is small
enough so that there exists $R$ such that
\[
D_{\xi}e^{2 \delta(R+ R^{2}/2+1)^{2}}\leq R
\]
where $D_{\xi}$ is the bound on the derivative of $\xi$ with respect to
$M_{T}$. In order to prove the nonexistence of solution for large $\delta$,
we need the following proposition.

\begin{proposition}\label{PDE1}
Consider the PDE of $g:[0,\infty)\times[0,1]\to\bbR$:
\begin{align*}
\partial_{t}g =\partial_{rr}g+\frac{1}{r}\partial_{r}g-\frac{\sin
  g\cos g}{r^{2}};\qquad g(0,r)=g_{0}(r), g(t,0)=0, \text{ and }g(t,1)=2\pi.
\end{align*}
This PDE admits a unique classical solution on $[0, T_{0})$ for some
$T_{0}\in\bbR_{+}$ and $ \lim_{T\to
  T_{0}}\partial_{r}g(t,0)=\infty$. % Moreover,
% \[
% \partial_{r}g(t,0)\geq C(T_{0}-t)^{-\frac{1}{1-\veps}}
% \]
% for some $C\in\bbR_{+}$.
\end{proposition}
\begin{proof}
See the
proof part (i) in Chang et al. (1992, \cite{Chang:1992p29958}).   
\end{proof}

Now, for $T_{0}$ in Proposition \ref{PDE1}, assume that $\delta>T_{0}$
and \eqref{counterex} has a solution
$(Y,Z)\in\bbH^{2}\times\bbH^{2}_{m}$ such that $Z$ is bounded. For $g(t,r)$ in
Proposition \ref{PDE1}, if we let $u:[0,T_{0})\times\bbR^{2}\to\bbR^{3}$ be
\[
u(t,x)=\begin{pmatrix}\frac{x_{1}}{|x|}\sin g(t,|x|)\\\frac{x_{2}}{|x|}\sin
g(t,|x|)\\ \cos g(t,|x|)\end{pmatrix},
\]
we can easily deduce that
\begin{align*}
Y_{t}=u(T-t, M_{t})\qquad\text{and}\qquad Z_{t}:=\nabla u(T-t,
  M_{t})\qquad t\in(T-T_{0},T]
\end{align*}
by using {\Ito} formula and the uniqueness of solution for BSDE. Note that
$\nabla u(T-t,0)\to\infty$ as $t\searrow T-T_{0}$ by Proposition
\ref{PDE1} and $(\nabla u)(t,x)$ is continuous in
$(t,x)$. Therefore, for any large $L$, there exists $\veps>0$ such
that $|\nabla u(T-t,x)|\geq L$ for all $(t,x)\in (T-T_{0},T-T_{0}+\veps)\times(-\veps,\veps)^{2}$. Since $M$ is a scaled Brownian motion starting at
$T-\delta$ and stopped at $\tau$,
\[
\bbP(|Z_{t}|\geq L \text{ for all }t\in (T-T_{0},T-T_{0}+\veps))\geq\bbP(M_{t}\in (-\veps,\veps)^{2}\text{ for all }t\in (T-T_{0},T-T_{0}+\veps))>0.
\] 
This implies  BSDE \eqref{counterex} cannot have a solution such that
$Z$ is bounded when $\delta\geq T_{0}$.
\begin{Remark}
Our counterexample above shows that there is no solution $(Y,Z)$ such
that $Z$ is bounded and therefore, Theorem \ref{localmulti} is sharp
in this sense. We do not exclude the possibility that
there is a solution $(Y,Z)\in\bbH^{2}\times\bbH^{2}_{m}$ such that $Z$
is not bounded. However,  if $\delta=T_{0}$ in above example,
 $\lim_{t\searrow T-T_{0}}Z_{t}=\infty$ almost surely.
\end{Remark}
\begin{Remark}
Another
 counterexample for the existence of solution for multidimensional
 quadratic BSDE is given by Frei and dos Reis (2012,
 \cite{Anonymous:2012p19201}). They proved the $Y$ part of solution
 explodes when the terminal condition is singular with respect to the
 perturbation of underlying martingale, i.e., Brownian motion. In our
 example, $Y$ is uniformly bounded and the terminal condition is
 smooth with bounded derivative. However, in our case, $Z$ explodes with positive
 probability.
\end{Remark}
\subsection{Existence and uniqueness when $d=1$.}
When $d=1$ and $m_{t}$ is invertible for
all $t\in[0,T]$, we can remove the smallness condition
on $K$ if we assume Lipschitzness of
$f(t,\gamma_{[0,t]},y,z)$ with respect to $y$. 
\begin{theorem}\label{onedim}
Assume that {\rm (M)} and {\rm (Diff')} hold, $d=1$, and $m_{t}$ is invertible for
all $t\in[0,T]$. In addition, assume that
\begin{itemize}
\item[{\rm (Loc')}] there exist $C_{y}\in\bbR_{+}$ and a nondecreasing function 
$\rho : \mathbb{R}_+ \to \mathbb{R}_+$ such that
$$
|f(t,\gamma_{[0,t]},y,z)-f(t,\gamma_{[0,t]},y',z')| \leq C_{y} |y-y'|+\rho(|z| \vee |z'|) |z-z'|
$$
for all $t \in[0,T]$, $y,y' \in \mathbb{R}^{d}$ and $z,z' \in
\mathbb{R}^{d\times n}$. 
\end{itemize}
Then, the BSDE
\[
Y_{t}=\xi(M_{[0,T]})+\int_{t}^{T}f(s,M_{[0,s]},Y_{s},Z_{s}m_{s})dA_{s}-\int_{t}^{T}Z_{s}dM_{s}
\]
has a unique solution $(Y,Z)\in\bbH^{2}\times \bbH_{m}^{2}$ such that $Z$
is bounded. Moreover, \[|Z_{t}|\leq \sqrt{n}\edg{\brak{D_{\xi}+\frac{D_{f}}{C_{y}}}e^{C_{y}K}-\frac{D_{f}}{C_{y}}},\; dt\otimes d\bbP\text{ -almost
everywhere.}\] In the case where $C_{y}=0$, the bound changes to
$\sqrt{n}(D_{\xi}+D_{f}K)$. If we assume {\rm (D)} and {\rm (Diff)} as well, $Y$ and $Z$ are
$\nabla$- and $\nabla^{m}$-differentiable, respectively, and $Z_{u}=\nabla_{u}Y_{u}$.
\end{theorem}
\begin{proof}
We already know that $(\xi,f)$ satisfies (S) by Lemma \ref{lmn:ZdY}.
Note that if we can prove the theorem under assumption (D), and (Diff),
we can generalize it to (Diff') using the same argument in Theorem
\ref{thm:ZdY}. Therefore, without loss of generality, we will assume
(D) and (Diff), and moreover, $D_\xi, C_{y}>0$. Let
\[
R:=\sqrt{n}\edg{\brak{D_{\xi}+\frac{D_{f}}{C_{y}}}e^{C_{y}K}-\frac{D_{f}}{C_{y}}}
\]
First let $(Y,Z)$ be the
solution of
\[
Y_{t}=\xi(M_{0,T})+\int_{t}^{T}g(s,M_{[0,s]},Y_{s},Z_{s}m_{s})dA_{s}-\int_{t}^{T}Z_{s}dM_{s}
\]
where $g$ is a smooth extension of $f$ such that
\begin{align*}
g(t,\gamma_{[0,t]},y,z):=\left\{\begin{array}{ll}f\brak{t,\gamma_{[0,t]},y,z}&\text{ if }|z|\leq R\\
f\brak{t,\gamma_{[0,t]},y,(R+1)z/|z|} &\text{ if }|z|\geq R+1
\end{array}
\right.
\end{align*}
and $|\partial_{z}g|\leq \rho(R+1)$. 
By our Proposition \ref{prop:diff} and \ref{thm:ZdY}, for $t\geq u$,
\begin{align}\label{smdiff}
\nabla_{u}^{\bf e_{i}}Y_{t}&=\Xi+\int_{t}^{T}\brak{\zeta_{s}+\eta_{s}\nabla_{u}^{\bf
  e_{i}}Y_{s}+\theta_{s} \cdot (\nabla_{u}^{{\bf e_{i}},m}Z)_{s}m_{s}} dA_{s}-\int_{t}^{T}(\nabla^{{\bf
  e_{i}},m}_{u}Z)_{s}dM_{s}
\end{align}
and $Z_{t}=\nabla_{t}Y_{t}$, 
where
\begin{align*}
\Xi&:=\nabla^{\bf e_{i}}_{u}\xi(M_{[0,T]})\\
 \zeta&:=(\nabla_{u}^{\bf e_{i}}g)(\cdot, M_{[0,\cdot]},y,z)\vline_{(y,z)=(Y_{\cdot},Z_{\cdot}m_{\cdot})}\\
\eta&:=(\partial_{y}g)(\cdot, M_{[0,\cdot]},Y_{\cdot},Z_{\cdot} m_{\cdot})\\
\theta&:=(\partial_{z}g)(\cdot, M_{[0,\cdot]},Y_{\cdot},Z_{\cdot} m_{\cdot})
\end{align*}
Let us compare \eqref{smdiff} with
\begin{align}\label{comp}
U_{t}&=D_{\xi}+\int_{t}^{T}\brak{D_{f}+C_{y}|U_{s}|+\rho(R+1)|V_{s}m_{s}|}
  dA_{s}-\int_{t}^{T}V_{s}dM_{s}\\\notag
\bar U_{t}&=-D_{\xi}+\int_{t}^{T}\brak{-D_{f}-C_{y}|\bar U_{s}|-\rho(R+1)|\bar V_{s}m_{s}|}
  dA_{s}-\int_{t}^{T}\bar V_{s}dM_{s}
\end{align}
The BSDEs \eqref{comp} have unique solutions in
$\bbH^{2}\times\bbH^{2}_{m}$ such that $U,\bar U\in\bbS^2$. Let us define
\begin{align*}
h(v)&:=\left\{\begin{array}{ll}
\frac{v^*}{|v|}&\text{if }v\neq 0\\
0&\text{otherwise}
\end{array}\right.\\
\frac{d\Gamma_t}{\Gamma_t}&=C_y h(U_s)dA_s+\rho(R+1)h(V_sm_s)^*m_s^{-1}dM_s;\qquad \Gamma_0=1.
\end{align*}
Here, we use $h$ as defined on either $\bbR$ or $\bbR^{1\times n}$ depending on the context. If we apply {\Ito} formula to $\Gamma_tU_t$,
\[
d(\Gamma_tU_t)=-D_f\Gamma_t
dA_t+\brak{\Gamma_tU_t\rho(R+1)h(V_tm_t)^*m_t^{-1}+\Gamma_tV_t}dM_t.\]
Since $h(U_t), h(V_tm_t)$ are bounded by $1$, by the same logic as in the proof of Theorem \ref{compthm}, $\Gamma\in\bbS^2$ and this implies that $\int\Gamma_s\rho(R+1)h(V_sm_s)^*m_s^{-1}dM_s$ and $\Gamma U+\int D_f\Gamma_sdA_s$ are true martingales. Therefore,
\begin{align*}
\Gamma_t&=\scE\brak{\int_0^\cdot C_y h(U_s)dA_s+\rho(R+1)\int_0^\cdot h(V_sm_s)^*m_s^{-1}dM_s}_t=\bbE\edg{\Gamma_T-C_y\int_t^T\Gamma_sh(U_s)dA_s\;\vline\;\scF_t}>0\\
U_t&=\frac{1}{\Gamma_t}\bbE\edg{\Gamma_TD_\xi+\int_t^TD_f\Gamma_sdA_s\;\vline\; \scF_t}.
\end{align*}
Since $U_t>0$, we have $h(U_s)=1$ and the first part implies that, since $\Gamma_t:=e^{C_yA_t}\scE(\rho(R+1)\int_0^\cdot h(V_sm_s)^*m_s^{-1}dM_s)_t$,
\[
\bbE\edg{\int_t^T\frac{\Gamma_s}{\Gamma_t}dA_s\;\vline\;\scF_t}\leq\frac{1}{C_y}\bbE\edg{\frac{\Gamma_T}{\Gamma_t}\;\vline\;\scF_t}-\frac{1}{C_y}\leq\frac{1}{C_y}e^{C_y(K-A_t)}-\frac{1}{C_y}.
\]
Therefore, we get
\[
U_t\leq\brak{D_\xi+\frac{D_f}{C_y}}e^{C_y(K-A_t)}-\frac{D_f}{C_y}.
\]
We can get the same upper bound for $-\bar U_t$ using the same argument. By the comparison theorem \ref{compthm}, we know \[|\nabla_{u}^{\bf
  e_{i}}Y_{t}|\leq \brak{D_{\xi}+\frac{D_{f}}{C_{y}}}e^{C_{y}(K-A_{t})}-\frac{D_{f}}{C_{y}}\leq\frac{R}{\sqrt{n}}
\] which implies $|Z_{t}m_{t}|\leq R$. Therefore,
$(Y,Z)$ is a solution of the original BSDE
\[
Y_{t}=\xi(M_{[0,T]})+\int_{t}^{T}f(s,M_{[0,s]},Y_{s},Z_{s}m_{s})dA_{s}-\int_{t}^{T}Z_{s}dM_{s}.
\]
Uniqueness can be easily checked by the same argument in the proof of
Theorem \ref{localmulti}.
\end{proof}

\section{Utility Maximization of Controlled SDE Driven by M}\label{sec:diff}
In this section, we apply the previous results to the utility maximization
problem for controlled SDEs driven by $M$. Our control is $\Delta$ and we require
\[
\Delta\in\scA:=\crl{X\in\bbH^{2}:\esssup_{(t,\omega)\in[0,T]\times\Omega}|X_{t}(\omega)|<\infty}
\]
For a given control $\Delta$,
consider one of the two SDEs driven by $M$ where $M$ satisfies (M):
\begin{align}\label{power}
X^{\Delta}_{t}&=x+\int_{0}^{t}X^{\Delta}_{s}b(s,M_{[0,s]},\Delta_{s})dA_{s}+\int_{0}^{t}X^{\Delta}_{s}\sigma(s,
  M_{[0,s]},\Delta_{s})dM_{s}\\ \label{exp}
X^{\Delta}_{t}&=x+\int_{0}^{t}b(s,M_{[0,s]},\Delta_{s})dA_{s}+\int_{0}^{t}\sigma(s,
  M_{[0,s]},\Delta_{s})dM_{s}
\end{align}
Here, 
$
b:[0,T]\times D\times\bbR^{1\times n}\to\bbR
                          \text{ and }\sigma:[0,T]\times D\times\bbR^{1\times
                          n}\to\bbR^{1\times n}
$
are jointly measurable functions such that $b(\cdot,
M_{[0,\cdot]},\Delta_{\cdot})$ and $\sigma(\cdot, M_{[0,\cdot]},\Delta_{\cdot})$
are predictable for any $\Delta\in\scA$.

\subsection{Power utility} 
Our objective in this subsection is to find a control $\Delta\in\scA$ that maximize
\[
\bbE\edg{ \frac{1}{\kappa}\brak{X^{\Delta}_{T}e^{-\xi(M_{[0,T]})}}^{\kappa}}
\]
where $\kappa\in(-\infty,0)\cup(0,1]$ and $X^{\Delta}$ is given by
\eqref{power}.
\begin{theorem}\label{app0} Assume that there exist an increasing
  continuous function
  $\rho:\bbR_{+}\to\bbR$ and a $\scB([0,T])\otimes\scD\otimes\scB(\bbR^{n})$-measurable function
  $k:[0,T]\times D\times\bbR^{1\times n}\to\bbR^{1\times n}$ that
  satisfy the following conditions:
\begin{itemize}
\item $|b(s,M_{[0,s]},\pi)|+|\sigma(s,M_{[0,s]},\pi)|\leq \rho(|\pi|)$ for all
  $\pi\in\bbR^{1\times n}$.
\item $|k(s,\gamma,z)|\leq \rho(|z|)$ for all
  $(s,\gamma,z)\in[0,T]\times D\times\bbR^{1\times n}$.
\item For all
$(s,\gamma,\pi,z)\in[0,T]\times D\times\scA\times\bbR^{d\times n}$,
the following inequality holds:
\begin{align*}
&G(s,\gamma,\pi,z):=-\half\kappa|\sigma(s,\gamma_{[0,s]},\pi) m_{s}-z|^{2}-b(s,\gamma_{[0,s]},\pi)+\half|\sigma(s,\gamma_{[0,s]},\pi) m_{s}|^{2}\\
&\geq G(s,\gamma_{[0,s]},k(s,\gamma_{[0,s]},z),z):=f(s,\gamma_{[0,s]},z)
%-\half\kappa|\sigma(s,\gamma_{[0,s]},k(s,\gamma_{[0,s]},z)) m_{s}-z|^{2}+b(s,\gamma_{[0,s]},k(s,\gamma_{[0,s]},z)) 
\end{align*}
\item The BSDE$(\xi,f)$ has a solution
$(Y,Z)\in\bbH^{2}\times\bbH^{2}_{m}$ such that $Z$ is bounded.
\end{itemize}
Then,
\[
\bar\Delta_{s}:=k(s,M_{[o,s]},Z_{s}m_{s})\in\scA
\]
is the optimal control and the optimal value is $\frac{x^{\kappa}e^{-\kappa Y_{0}}}{\kappa}$.
\end{theorem}
\begin{proof}
First of all, since $M_{[0,\cdot]}:[0,T]\times\Omega\to D$ is an
  adapted continuous function, note that $\bar\Delta$ is a predictable
  process because it is a deterministic measurable function of predictable
  processes. Moreover, since $Z$ is bounded, $\bar\Delta$ is bounded by
  our assumption. Therefore, $\bar\Delta\in\scA$. 

As in Hu et al. (2005, \cite{Hu:2005tj}), we use the martingale
technique to prove the theorem. Note that,  since
$b(s,M_{[0,s]},\Delta_{s})$ and $\sigma(s,M_{[0,s]},\Delta_{s})$ are bounded for $\Delta\in\scA$,
\eqref{power} has a unique strong solution
\[
X^{\Delta}_{t}=x\exp\brak{\int_{0}^{t}\brak{b(s,M_{[0,s]},\Delta_{s})-\half|\sigma(s,M_{[0,s]},\Delta_{s})m_{s}|^{2}}dA_{s}+\int_{0}^{t}\sigma(s,M_{[0,s]},\Delta_{s})dM_{s}}.
\]
For notational convenience, let us use
$b^{\Delta}_{s}:=b(s,M_{[0,s]},\Delta_{s})$ and $\sigma^{\Delta}_{s}:=\sigma(s,M_{[0,s]},\Delta_{s})$.
Let us define a family of stochastic process $\crl{U^{\Delta}}_{\Delta\in\scA}$ given
by
\begin{align*}
&U^{\Delta}_{t}= \frac{1}{\kappa}(X^{\Delta}_{t}e^{-Y_{t}})^{\kappa}\\
&=\frac{x^{\kappa}}{\kappa}\exp\brak{-\kappa Y_{0}+\kappa\int_{0}^{t}\brak{b^{\Delta}_{s}-\half|\sigma^{\Delta}_{s}m_{s}|^{2}+f(s,
  M_{[0,s]},Z_{s}m_{s})}dA_{s}+\kappa\int_{0}^{t}\brak{\sigma^{\Delta}_{s}-Z_{s}}dM_{s}}\\
&=\frac{x^{\kappa}e^{-\kappa Y_{0}}}{\kappa}\scE\brak{\kappa\int_{0}^{\cdot}\brak{\sigma^{\Delta}_{s}-Z_{s}}dM_{s}}_{t}\exp\brak{\kappa\int_{0}^{t}\brak{b^{\Delta}_{s}-\half|\sigma^{\Delta}_{s}m_{s}|^{2}+\frac{\kappa}{2}|\sigma^{\Delta}_{s}m_{s}-Z_{s}m_{s}|^{2}+f(s,
  M_{[0,s]},Z_{s}m_{s})}dA_{s}}\\
&=\frac{x^{\kappa}e^{-\kappa Y_{0}}}{\kappa}\scE\brak{\kappa\int_{0}^{\cdot}\brak{\sigma^{\Delta}_{s}-Z_{s}}dM_{s}}_{t}\exp\brak{\kappa\int_{0}^{t}\brak{f(s,
  M_{[0,s]},Z_{s}m_{s})-G(s,M_{[0,s]},\Delta_{s},Z_{s}m_{s})}dA_{s}}
\end{align*}
Since
\begin{align*}
&\edg{-\int_{0}^{\cdot}\kappa\brak{\sigma^{\Delta}_{s}-Z_{s}}dM_{s},-\int_{0}^{\cdot}\kappa\brak{\sigma^{\Delta}_{s}-Z_{s}}dM_{s}}_{T}\leq\kappa^{2}\int_{0}^{T}\abs{(\sigma^{\Delta}_{s}-Z_{s})m_{s}}^{2}dA_{s}
\end{align*}
and
\[
\abs{(\sigma^{\Delta}_{s}-Z_{s})m_{s}}\leq \rho(|\Delta_{s}|)+|Z_{s}|,
\]
$\int_{0}^{\cdot}\kappa\brak{\sigma(s,M_{[0,s]},\Delta_{s})-Z_{s}}dM_{s}$
is a BMO martingale for each $\Delta\in\scA$. Therefore, \[\scE\brak{-\int_{0}^{\cdot}\kappa\brak{\sigma(s,M_{[0,s]},\Delta_{s})-Z_{s}}dM_{s}}_{t}\]
is a true martingale: see Kazamaki (1994, \cite{Kazamaki:1994ux}). 
Moreover, note that $\frac{1}{\kappa}e^{\kappa x}$ is an increasing
function and
\[
f(s,M_{[0,s]},Z_{s}m_{s})-G(s,M_{[0,s]},\Delta_{s},Z_{s}m_{s})\leq 0
\]
where equality holds when $\Delta_{s}=\bar\Delta_{s}=k(s,
M_{[0,s]},Z_{s}m_{s})$. Therefore, $U^{\Delta}$ are
supermartingales and $U^{\bar\Delta}$ is a true martingale. This implies
that
\[
\bbE\edg{ \frac{1}{\kappa}(X^{\Delta}_{T}e^{-\xi(M_{[0,T]})})^{\kappa}}\leq U^{\Delta}_{0}=\frac{x^{\kappa}e^{-\kappa Y_{0}}}{\kappa} =U^{\bar\Delta}_{0}=\bbE\edg{ \frac{1}{\kappa}(X^{\bar\Delta}_{T}e^{-\xi(M_{[0,T]})})^{\kappa}}.
\]
Therefore, $\bar\Delta$ is the optimal control and the claim is proved.
\end{proof}
\subsection{Exponential utility}
In this subsection, our objective is to find a control $\Delta\in\scA$ that maximize
\[
\bbE\edg{ -\exp(-\kappa(X^{\Delta}_{T}-\xi(M_{[0,T]})))}
\]
where $\kappa>0$ and $X^{\Delta}$ is given by \eqref{exp}.

\begin{theorem}\label{app} Assume that there exist $C\in\bbR_{+}$, an increasing
  continuous function
  $\rho:\bbR_{+}\to\bbR$, and a $\scB([0,T])\otimes\scD\otimes\scB(\bbR^{n})$-measurable function
  $k:[0,T]\times D\times\bbR^{1\times n}\to\bbR^{1\times n}$ that
  satisfy the following conditions:
\begin{itemize}
\item $X^{\Delta}$ is well-defined for any $\Delta\in\scA$.
\item $|\sigma(s,M_{[0,s]},\pi)|\leq \rho(|\pi|)$ for all
  $\pi\in\bbR^{1\times n}$.
\item $|k(s,\gamma,z)|\leq \rho(|z|)$ for all
  $(s,\gamma,z)\in[0,T]\times D\times\bbR^{1\times n}$.% and $k(\cdot,M_{[0,\cdot]},X_{\cdot})$ is  predictable for $X\in\bbH^{2}_{m}$.
\item For all
$(s,\gamma,\pi,z)\in[0,T]\times D\times\scA\times\bbR^{d\times n}$,
the following inequality holds:
\begin{align*}
&G(s,\gamma,\pi,z):=\half\kappa|\sigma(s,\gamma_{[0,s]},\pi) m_{s}-z|^{2}
-b(s,\gamma_{[0,s]},\pi)\geq G(s,\gamma_{[0,s]},k(s,\gamma_{[0,s]},z),z):=f(s,\gamma_{[0,s]},z)
%-\half\kappa|\sigma(s,\gamma_{[0,s]},k(s,\gamma_{[0,s]},z)) m_{s}-z|^{2}+b(s,\gamma_{[0,s]},k(s,\gamma_{[0,s]},z)) 
\end{align*}
\item The BSDE$(\xi,f)$ has a solution
$(Y,Z)\in\bbH^{2}\times\bbH^{2}_{m}$ such that $Z$ is bounded.
\end{itemize}
Then,
\[
\bar\Delta_{s}:=k(s,M_{[o,s]},Z_{s}m_{s})\in\scA
\]
is the optimal control and the optimal value is $-e^{-\kappa(x-Y_{0})}$.
\end{theorem}
\begin{proof} By the same argument in the proof of Theorem \ref{app0},
  we can easily show that $\bar\Delta\in\scA$. 

For $\Delta\in\scA$, we define a family of stochastic process
$\crl{U^{\Delta}}$ given by
\begin{align*}
U^{\Delta}_{t}=-\exp\brak{-\kappa(X^{\Delta}_{t}-Y_{t}) }.
\end{align*}
For notational convenience, let us use
$b^{\Delta}_{s}:=b(s,M_{[0,s]},\Delta_{s})$ and $\sigma^{\Delta}_{s}:=\sigma(s,M_{[0,s]},\Delta_{s})$.
Note that
we have
\begin{align*}
&U^{\Delta}_{t}=-\exp\brak{-\kappa\brak{x-Y_{0}+\int_{0}^{t}\brak{b^{\Delta}_{s}+f(s,
  M_{[0,s]},Z_{s}m_{s})}dA_{s}+\int_{0}^{t}\brak{\sigma^{\Delta}_{s}-Z_{s}}dM_{s}}
             }\\
&=-e^{-\kappa(x-Y_{0})}\scE\brak{-\int_{0}^{\cdot}\kappa\brak{\sigma^{\Delta}_{s}-Z_{s}}dM_{s}}_{t}\exp\brak{\kappa\int_{0}^{t}\brak{G(s,M_{[0,s]},\Delta_{s},Z_{s}m_{s})-f(s,M_{[0,s]},Z_{s}m_{s})}dA_{s}}
\end{align*}
By the same logic in the proof of Theorem \ref{app0},
$\scE\brak{-\int_{0}^{\cdot}\kappa\brak{\sigma^{\Delta}_{s}-Z_{s}}dM_{s}}$
is a true martingale.
Since
\begin{align*}
G(s,M_{[0,s]},\Delta_{s},Z_{s}m_{s})\geq f(s,M_{[0,s]},Z_{s}m_{s})
\end{align*}
and equality holds when
$\Delta_{s}=\bar\Delta_{s}=k(s,M_{[0,s]},Z_{s}m_{s})$, $U^{\Delta}$ are
supermartingales and $U^{\bar\Delta}$ is a true martingale. This implies
that
\[
\bbE\edg{-\exp\brak{-\kappa(X^{\Delta}_{T}-Y_{T}) }}\leq U^{\Delta}_{0}= -e^{-\kappa(x-Y_{0})}=U^{\bar\Delta}_{0}=\bbE\edg{-\exp\brak{-\kappa(X^{\bar\Delta}_{T}-Y_{T}) }}.
\]
Therefore, $\bar\Delta$ is the optimal control and the claim is proved.
\end{proof}
\section{Application to Optimal Portfolio Selection}\label{sec:ops}
In this section, we apply above result to a
market consists of 1 risk-free asset with no interest and $n$ risky
assets $S^{1},...,S^{n}$ whose prices follow dynamics
\[
\frac{dS^{i}_{t}}{S^{i}_{t}}=\theta^{i}dA_{t}+dM^{i}_{t}.
\]
where $\theta\in\bbR^{n\times 1}$ and $M$ satisfies (M) and (Diff'). In addition
we always assume that $m_{s}$ is invertible and deterministic and $\xi$ satisfies (Diff')
.
\subsection{Power utility function}
In this case, the $\Delta^{i}_{t}$ represent the portion of
wealth invested on asset $S^{i}$ at time $t$. We also assume
that, if the investor invest $\Delta$
on time $[t,t+dt]$, there is cost $X^{\Delta}_tp(\Delta_{t})dt$
for some function $p:\bbR\to\bbR$ which we call a penalty function. Then, if the initial wealth is $x$, the  wealth process associated with $\Delta$ is given by
\begin{align*}
dX^{\Delta}_{t}=X^{\Delta}_{t}(-p(\Delta_{t})+\Delta_{t}\theta)dA_{t}+X^{\Delta}_{t}\Delta_{t}dM_{t};\qquad
  X^{\Delta}_{0}=x.
\end{align*}
The investor tries to maximize
\[
\bbE\edg{ \frac{1}{\kappa}(X^{\Delta}_{T}e^{-\xi(M_{[0,T]})})^{\kappa}}
\]
where $\kappa\in(-\infty, 0)\cup(0,1]$. Then, the function $G$ and $k$ in Theorem \ref{app0} are given by
\begin{align*}
G(s,\gamma,\pi,z)&:=-\half\kappa|\pi m_{s}-z|^{2}+p(\pi)-\pi\theta+\half|\pi m_{s}|^{2}\\
k(s,\gamma,z)&\in{\rm argmin}_{\pi\in\scA}G(s,\gamma,\pi,z).
\end{align*}
Let us provides three specific examples. The first example is the case
where the investment is restricted to some closed set. This is
similar to Hu et al. (2005, \cite{Hu:2005tj}), but we are seeking the optimal
investment that does not require infinite investment.

\begin{Example}[Incomplete Market]
For an investor, there may be a restriction on the investment strategy
$\Delta$ so that $\Delta_{t}(\omega)\in\scC$ for some closed set $\scC\subset\bbR^{1\times n}$, then we can define
  $p(\pi)=0$ when $\pi\in \scC$ and $\infty$ otherwise. When $\kappa<1$, the minimum
  of $G$ is attained when 
\[
\pi=\Pi_{s}\brak{\frac{\theta^{*}-\kappa z}{1-\kappa}m_s^{-1}}:=k(s,\gamma,z)
\]
where $\Pi_{s}$
  is the projection to the set $\scC m_{s}:=\crl{xm_{s}:x\in
    \scC}$. Since $k$ does not depend on $\gamma$ and Lipschitz in $z$, 
$f(s,z):=G(s,\gamma,k(s,\gamma,z),z)$ satisfies (Diff') and (Loc'), by Theorem \ref{onedim}, BSDE($\xi,f$)
has a unique solution $(Y,Z)\in\bbH^{2}\times\bbH^{2}_{m}$ such that
$Z$ is bounded. Therefore, by Theorem \ref{app0}, the optimal control
is
\[
\bar\Delta_s:=\Pi_{s}\brak{\frac{\theta^{*}m_s^{-1}-\kappa Z_{s}}{1-\kappa}}
\]
and the optimal value is $\frac{x^{\kappa}e^{-\kappa Y_{0}}}{\kappa}$.\\
On the other hand, if $\kappa=1$ and $\scC$ is a bounded set,
\[
G(s,\gamma,\pi,z):=\pi m_{s}(z^{*}-m_{s}^{-1}\theta)-\half |z|^{2}+p(\pi)
\]
attains minimum $f(s,z)$ when $\pi m_{s}(z^{*}-m_{s}^{-1}\theta)$ is
minimized over $\pi\in\scC$. Note that $f(s,z)$ does not depend on $\gamma$ and
satisfies (Loc') where $\rho$ has linear growth. Therefore, BSDE($\xi,f$) has a unique solution
$(Y,Z) \in\bbH^{2}\times\bbH^{2}_{m}$ such that $Z$ is bounded and the optimal control is
\[
\bar\Delta_{s}:={\rm arg}\min_{\pi\in\scC}(\pi m_{s}(m_{s}^{*}Z^{*}_{s}-m_{s}^{-1}\theta))
\]
and  the optimal value is $\frac{x^{\kappa}e^{-\kappa Y_{0}}}{\kappa}$.
\end{Example}

The next example is the case where the penalty function encourage the
diversification of portfolio among risky assets. The penalty function
is given by $p(\pi):=\abs{\pi(\bbI-w)}^{\beta}$ where $\beta>1$ and $w$ is a $n$-by-$n$ matrix where
  $(\bbI-w)(\bbI-w)^{*}+\frac{1-\kappa}{2}m_{s}m_{s}^{*}$ is
  invertible. Note that
  $p$ attains minimum when $\pi=\pi w$, so we can think $\pi w$ is the
  preferred distribution of wealth. If one want to encourage
  diversification of portfolio, one can set
  $w^{ij}\in[1/n-\veps,1/n+\veps],\forall i,j$, for small
  $\veps>0$. In the extreme case, if one set $\veps=0$, the penalty attains minimum when $\pi^{i}=\pi^{j}$ for all $i,j=1,...,n$. However, in this case,
  $\bbI-w$ is not invertible so we encounter problem when $\kappa=1$.
The bigger $\beta$ implies the stronger encouragement toward the
preferred wealth distribution. 

\begin{Example}\label{ex30}[Diversification of Portfolio] Assume that
\[
p(\pi):=\abs{\pi(\bbI-w)}^{\beta}.
\] 
Note that $p$ is a convex function which is
  differentiable everywhere and therefore,
  $G$ has a minimum when it satisfies the first order condition:
\[
\nabla p(\pi)+\brak{1-\kappa} \pi m_{s}m_{s}^{*}=-\kappa zm_{s}^{*}+\theta^{*}.
\]
(Case $\kappa<1$) To find an explicit solution, let us consider the case where
$\beta=2$ and $\kappa<1$. Then, the first order condition becomes
\[
2\pi\brak{\bbI-w}\brak{\bbI-w}^{*}+(1-\kappa) \pi m_{s}m_{s}^{*}=-\kappa zm_{s}^{*}+\theta^{*}
\]
and this implies
\[
\pi=k(s,z):=\half(-\kappa zm_{s}^{*}+\theta^{*})\brak{(\bbI-w)(\bbI-w)^{*}+\frac{1-\kappa}{2}m_{s}m_{s}^{*}}^{-1}.
\]
Let $f(s,z):=G(s,\gamma,k(s,z),z)$. Since $k$ is an affine function of
$z$, $f$ satisfies (Loc') with $\rho$ that has linear growth, and the BSDE($\xi,f$)
has a unique solution $(Y,Z)\in\bbH^{2}\times\bbH^{2}_{m}$ such that $Z$ is bounded. Therefore, the
optimal control is
\[
\bar \Delta_{s}=\half(-\kappa Z_{s}m_{s}m_{s}^{*}+\theta^{*})\brak{(\bbI-w)(\bbI-w)^{*}+\frac{1-\kappa}{2}m_{s}m_{s}^{*}}^{-1}.
\]
(Case $\kappa=1$) On the other hand, let us assume $\kappa=1$ which means that the
investor is risk-neutral but there is a penalty if the portfolio is away from $w$. Then the first order condition becomes
\[
\beta\abs{\pi\brak{\bbI-w}}^{\beta-2}\pi\brak{\bbI-w}(\bbI-w)^{*}=-zm_{s}^{*}+\theta^{*}
\]
and this implies
% \[
% \pi\brak{\bbI-w}=\brak{\frac{|(-zm_{s}^{*}+\theta^{*})((\bbI-w)^{*})^{-1}|^{2-\beta}}{\beta}}^{1/(\beta-1)}(-zm_{s}^{*}+\theta^{*})((\bbI-w)^{*})^{-1}
% \] Therefore, 
\[
k(s,z)=\brak{\frac{|(-zm_{s}^{*}+\theta^{*})((\bbI-w)^{*})^{-1}|^{2-\beta}}{\beta}}^{1/(\beta-1)}(-zm_{s}^{*}+\theta^{*})((\bbI-w) (\bbI-w)^{*})^{-1}
\]
Therefore,
\[
f(s,z):=\abs{\frac{(-zm_{s}^{*}+\theta^{*})( (\bbI-w)^{*})^{-1}}{\beta}}^{\frac{\beta}{\beta-1}}-k(s,z
)(\theta-m_{s}z^*)-\half|z|^{2}
\]
Note that when $\beta\geq2$, (Loc') is satisfied $\rho$ with linear
growth. On the other hand, if $1<\beta<2$, (Loc') is satisfied with
$\rho$ that has superlinear growth. Therefore, in any case,
BSDE($(\xi,f)$) has a unique solution $(Y',Z')$ such that $Z'$ is
bounded, and we have optimal control
\[
\bar\Delta=\brak{\frac{|(-Z'_{s}m_{s}m_{s}^{*}+\theta^{*})((\bbI-w)^{*})^{-1}|^{2-\beta}}{\beta}}^{1/(\beta-1)}(-Z'_{s}m_{s}m_{s}^{*}+\theta^{*})((\bbI-w) (\bbI-w)^{*})^{-1}
\]
\end{Example}
\begin{Remark}
In above example, if $\kappa=1$ and $\beta<2$, the driver $f(s,z)$ of BSDE has
superquadratic growth in $z$. As in Delbaen et al. (2010, \cite{Delbaen:2010gj}), this BSDE may not have a solution or may have
infinite number of solution if we do not assume (Diff'). Due to this
difficulty, as far as the author knows, above example is the first financial application of superquadratic BSDE.
\end{Remark}
The last example is when there is information processing cost for
trading assets. Any investment on an asset requires constant follow-up
of information. It may be possible that this information
processing cost is so high that it is more reasonable not to trade the
asset at all.
\begin{Example}[Information Processing Cost]
Assume that $\theta=0$ and $m_{s}:=n^{-\half}\bbI$ where $\bbI$ is
  the $n$-by-$n$ identity matrix. For each $i=1,...,n$, risky assets $S^{i}$ may require  information processing cost $C_{i}$ if the investor has nonzero amount of asset $S^{i}$. This implies $p(\pi)=\sum_{i=1}^{n}C_{i}
  \indicator{\crl{\pi^{i}\neq 0}}$. Then,
\begin{align*}
G(s,\gamma,\pi,z)&=\frac{1-\kappa}{2n}|\pi|^{2}+\sum_{i=1}^{n}C_{i}
  \indicator{\crl{\pi^{i}\neq 0}}+\frac{\kappa}{\sqrt{n}}\pi z^{*}-\frac{\kappa}{2}|z|^{2}\\
 &=\frac{1-\kappa}{2n}\sum_{i=1}^{n}\edg{\brak{\pi^{i}+\frac{\kappa\sqrt{n}}{1-\kappa}z^{i}}^{2}+C_{i}
  \indicator{\crl{\pi^{i}\neq 0}}-\frac{\kappa n}{(1-\kappa)^{2}}|z^{i}|^{2}}
\end{align*}
Therefore, $G$ is minimized when
\[
\pi^{i}=k^{i}(s,z):=\left\{\begin{array}{ll} -\frac{\kappa\sqrt{n}}{1-\kappa}z^{i}&\text{when
                                            }|z^{i}|>\frac{1-\kappa}{|\kappa|}\sqrt{\frac {C_{i}}{n}}\\ 0&\text{otherwise}\end{array}\right.
\]
for each $i=1,...,n$. If we define
\[
f(s,z):=G(s,\gamma,k(s,z),z)=\frac{1-\kappa}{2n}\sum_{i=1}^{n}\min\brak{-\frac{\kappa n}{(1-\kappa)}|z^{i}|^{2},C_{i}-\frac{\kappa n}{(1-\kappa)^{2}}|z^{i}|^{2}}
\]
Since $f$ is satisfies (Loc') where $\rho$ has linear growth, the BSDE($\xi,f$) has a unique solution $(Y,Z)$ such that $Z$ is bounded. Therefore, the
optimal control is
\[
\bar\Delta^{j}_{s}=\left\{\begin{array}{ll} -\frac{\kappa}{1-\kappa}Z_{s}^{i}&\text{when
                                            }|Z_{s}^{i}|>\frac{1-\kappa}{|\kappa|}\sqrt{C_{i}}\\ 0&\text{otherwise}\end{array}\right.
\]
\end{Example}

\bigskip

\subsection{Exponential utility function}
At time
$t\in[0,T]$, an investor
invest $\Delta^{i}_{t}$ unit of currency on the asset $S^{i}$.  If the investor invest $\Delta$
on time $[t,t+dt]$, there is cost $p(\Delta_{t})dt$
for some function $p:\bbR\to\bbR$ which we call a penalty function.
Then, if the initial wealth is $x$, the  wealth process associated with $\Delta$ is given by
\begin{align*}
dX^{\Delta}_{t}=(-p(\Delta_{t})+\Delta_{t}\theta)dA_{t}+\Delta_{t}dM_{t};\qquad
  X^{\Delta}_{0}=x.
\end{align*}
Assume that the investor trying to maximize
\[
\bbE \edg{-\exp\brak{- \kappa (X^{\Delta}_{T}-\xi(M_{[0,T]}))}}
\]
for some $\kappa>0$. Then, $G$ and
$k$ in Theorem \ref{app} are given by
\begin{align*}
G(s,\gamma,\pi,z)&:=\half\kappa|\pi m_{s}-z|^{2}+p(\pi)-\pi\theta \\
k(s,\gamma,z)&\in{\rm argmin}_{\pi\in\scA}G(s,\gamma,\pi,z)
\end{align*}
and we can apply Theorem \ref{app} if the assumptions are satisfied.

The followings are three examples which are analogous to the power utility
examples in the previous sections.

\begin{Example}[Incomplete Market]
We assume that $\theta=0$ and the investment $\Delta_{t}(\omega)$
should be in some closed set $\scC\subset\bbR^{1\times n}$. In this case, we set
  $p(\pi)=0$ when $\pi\in \scC$ and $\infty$ otherwise. Then, the minimum
  of $G$ is attained when $\pi=(\Pi_{s}z)m_{s}^{-1}$ where $\Pi_{s}$
  is the projection to the set $\scC m_{s}$. Since $f(s,z):=\half\kappa|\Pi_{s}z-z|^{2}$
satisfies (Diff') and (Loc'), by Theorem \ref{onedim}, BSDE($\xi,f$)
has a unique solution $(Y,Z)$ such that
$Z$ is bounded. Therefore, by Theorem \ref{app}, the optimal control
is
$\bar\Delta_{s}:=\Pi_{s}(Z_{s}m_{s})=\Pi (Z_{s})$
where $\Pi$ is the projection to set $\scC$. 
\end{Example}

\begin{Example}[Diversification of Portfolio] Let $w$ be an $n$-by-$n$
  matrix and assume that
  $(\bbI-w)(\bbI-w)^{*}+\kappa
  m_{s}m_{s}^{*}/2$ is invertible for all $s\in[0,T]$. We set
$p(\pi):=\abs{\pi(\bbI-w)}^{2}$.
Note that $p$ is a convex function with
  differentiable everywhere and therefore,
  $G$ has a minimum when
\[
\nabla p(\pi)+\kappa \pi m_{s}m_{s}^{*}=\kappa zm_{s}^{*}+\theta^{*}.
\]
Therefore, $G$ is minimized when
\begin{align*}
\pi&=\half (\kappa zm_{s}^{*}+\theta^{*})\edg{(\bbI-w)(\bbI-w)^{*}+\half\kappa
  m_{s}m_{s}^{*}}^{-1}=:k(s,z). \end{align*}
Let $f(s,z):=G(s,\gamma,k(s,z),z)$. Since $k$ is a linear function of $z$, the BSDE($\xi,f$)
has a unique solution $(Y,Z)$ such that $Z$ is bounded. Therefore, the
optimal control is
\[
\bar\Delta_{s}=\half (\kappa Z_{s}m_{s}m_{s}^{*}+\theta^{*})\edg{(\bbI-w)(\bbI-w)^{*}+\half\kappa
  m_{s}m_{s}^{*}}^{-1}.
\]
\end{Example}

\begin{Example}[Information Processing Cost]
Assume that $\theta=0$ and $m_{s}:=n^{-\half}\bbI$ where $\bbI$ is
  the $n$-by-$n$ identity matrix. For each $i=1,...,n$, risky assets $S^{i}$ may require  information processing cost $C_{i}$ if the investor decide to trade
  $S^{i}$. This implies $p(\pi)=\sum_{i=1}^{n}C_{i}
  \indicator{\crl{\pi^{i}\neq 0}}$. Then,
\[
G(s,\gamma,\pi,z)=\half\kappa\sum_{j=1}^{n}\edg{\abs{\frac{\pi^{j}}{\sqrt{n}}-z^{j}}^{2}+\frac{2 C_{j}}{\kappa}\indicator{\crl{\pi^{j}\neq
    0}}}
\]
Therefore, $G$ is minimized when
\[
\pi^{j}=k^{j}(s,z):=\left\{\begin{array}{ll} \sqrt{n}z^{j}&\text{when
                                            }|z^{j}|>\sqrt{\frac {2C_{j}}{\kappa}}\\ 0&\text{otherwise}\end{array}\right.
\]
for each $j=1,...,n$. If we define
\[
f(s,z):=G(s,\gamma,k(s,z),z)=\half\kappa\sum_{j=1}^{n}\min\brak{|z^{j}|^{2},\frac {2C_{j}}{\kappa}}
\]
Since $f$ is Lipschitz in $z$, the BSDE($\xi,f$) has a unique solution $(Y,Z)$ such that $Z$ is bounded. Therefore, the
optimal control is
\[
\bar\Delta^{j}_{s}=\left\{\begin{array}{ll} Z_{s}^{j}&\text{when
                                            }|Z_{s}^{j}|>\sqrt{\frac {2nC_{j}}{\kappa}}\\ 0&\text{otherwise}\end{array}\right.
\]
\end{Example}

\appendix
\section{Appendix}
\begin{lemma}\label{mtg}
Assume $Y\in\bbS^{2}, Z\in\bbH^{2}_{m}$, and $N\in\bbM^{2}$. Then, $\int
Y^{*}_{s-}Z_{s}dM_{s}, \int Y_{s-}dN_{s}$ are true martingales.
\end{lemma}
\begin{proof}
Note that, by Cauchy-Schwartz inequality,
\begin{align*}
\bbE&\edg{\int_{0}^{\cdot}Y^{*}_{s-}Z_{s}dM_{s},\int_{0}^{\cdot}Y^{*}_{s-}Z_{s}dM_{s}}^{\half}_{T}\\
&=\bbE \sqrt{\int_{0}^{T}|Y^{*}_{s-}Z_{s}m_{s}|^{2}dA_{s}}
\leq\bbE\sup_{t\in[0,T]}|Y_{t}|\sqrt{\int_{0}^{T}|Z_{s}m_{s}|^{2}dA_{s}}\leq\norm{Y}_{\bbS^{2}}^{2}\norm{Z}_{\bbH^{2}_{m}}<\infty
\end{align*}
and
\begin{align*}
 \bbE&\edg{\int_{0}^{\cdot}Y^{*}_{s-}dN_{s},\int_{0}^{\cdot}Y^{*}_{s-}dN_{s}}^{\half}_{T}\leq\bbE\sqrt{\int_{0}^{T}|Y_{s-}|^{2}d{\rm tr}[N,N]_{s}}
\leq\bbE\sup_{t\in[0,T]}|Y_{t}|\sqrt{{\rm tr}[N,N]_{T}}\leq\norm{Y}_{\bbS^{2}}^{2}\norm{N}_{\bbM^{2}}<\infty.
\end{align*}
Then, by Burkholder-Davis-Gundy inequality,
$\sup_{t\in[0,T]}\int_{0}^{t}Y^{*}_{s-}Z_{s}dM_{s}$ and
$\sup_{t\in[0,T]}\int_{0}^{t}Y_{s-}^{*}dN_{s}$ are integrable and
therefore, $\int_{0}^{\cdot}Y^{*}_{s-}Z_{s}dM_{s}$ and $\int_{0}^{\cdot}Y^{*}_{s-}dN_{s}$ are true martingales.
\end{proof}

\begin{lemma}
A $\bbR^{k}$-valued stochastic process $X$ is adapted to $\bbF^{M'}$ if and only if there
exists a function $\scX:[0,T]\times D\to\bbR^{k}$ such that
\[
X_{t}=\scX(t,M'_{[0,t]})
\]
holds almost surely for each $t\in[0,T]$ and $\scX(t,\cdot)$ is $\scD$-measurable.
\end{lemma}
\begin{proof}
Note that if $X$ is adapted to $\bbF^{M'}$, there exists $X'$ which is adapted to the filtration generated by
$M$ and $X=X'$ almost surely. Therefore, without loss of generality, we can prove the theorem under the filtration generated by $M'$.\\
($\Rightarrow$) By Cinlar (2010, \cite{Cinlar:2011ei},
Chapter 2, Proposition 4.6), for all $t\in[0,\infty)$,
there exists a sequence $(t_{n})_{n=1,2,...}\subset [0,t]$ and a
$\otimes_{l}\scB(\bbR^{n})$-measurable function $\scX_{t}^{c}:\prod_{l}\bbR^{n}\to\bbR^{k}$ such
that
\[
X_{t}=\scX^{c}_{t}(M'_{t_{1}},M'_{t_{2}},...).
\]
Let us denote $\pi$ to be the coordinate map; that is,
$\pi_{a,b,...}(\gamma)=(\gamma_{a},\gamma_{b},...)$. Then,
\[
X_{t}:=(\scX^{c}_{t}\circ\pi_{t_{1},t_{2},...})(M'_{[0,t]})
\]
Since $\pi_{t_{1},t_{2},...}$ is $\scD$-measurable, $\scX(t,\cdot):=\scX^{c}_{t}\circ\pi_{t_{1},t_{2},...}$ is the function
we want.\\
($\Leftarrow$) Since $M'_{[0,t]}$ is
$\scF_{t}\backslash\scD$-measurable and $\scX$ is
$\scD$-measurable, the claim is obvious.
\end{proof}

\end{document}